\documentclass[a4paper]{article}

\usepackage{amscd,amssymb,amsmath,amsthm}
\usepackage{graphicx}
\usepackage[backend=biber,style=ieee,giveninits=true,sorting=anyt]{biblatex}
\usepackage{bbm}
\usepackage{hyperref}
\usepackage{color}
\usepackage{todonotes}
\usepackage[section]{placeins}
\usepackage[shortlabels]{enumitem}
\usepackage[lofdepth,lotdepth]{subfig}
\usepackage{tikz}
\usetikzlibrary{arrows,shapes,decorations,automata,backgrounds,petri,bending,trees,decorations.pathreplacing}

\newcommand{\Z}{\mathbb{Z}}
\newcommand{\R}{\mathbb{R}}
\newcommand{\N}{\mathbb{N}}
\newtheorem{thm}{Theorem}[section]
\newtheorem{defn}[thm]{Definition}
\newtheorem{lemma}[thm]{Lemma}
\newtheorem{pro}[thm]{Proposition}
\newtheorem{rk}[thm]{Remark}
\newtheorem{cor}[thm]{Corollary}

\newtheorem{ex}[thm]{Example}

\newcommand{\footnoteremember}[2]{
	\footnote{#2}
	\newcounter{#1}
	\setcounter{#1}{\value{footnote}}
}
\newcommand{\footnoterecall}[1]{
	\footnotemark[\value{#1}]
}
\def\violet#1{{\textcolor{violet}{#1}}}

\begin{document}
\title{Infinite-volume states with irreducible localization sets for gradient models on trees}
\author{
 Alberto Abbondandolo,\footnoteremember{RUB}{Ruhr-Universit\"at   Bochum, Fakult\"at f\"ur Mathematik, D44801 Bochum, Germany} \footnote{Alberto.Abbondandolo@ruhr-uni-bochum.de}
	\and 
	Florian Henning,\footnoterecall{RUB} \footnote{Florian.Henning@ruhr-uni-bochum.de}
	\and  Christof K\"ulske, \footnoterecall{RUB} \footnote{Christof.Kuelske@ruhr-uni-bochum.de}
\and  Pietro Majer \footnoteremember{Pisa}{Universit{\`a} di Pisa, Dipartimento di Matematica, I56127 Pisa, Italy} \footnote{Pietro.Majer@dm.unipi.it}\, 
	\,  
}
\date{February 10, 2023} 	
\maketitle
	
\begin{abstract} {We consider general classes of gradient models 
on regular trees with values in a countable Abelian group $S$ such as $\Z$ or $\Z_q$, in regimes of strong coupling (or low temperature). This includes unbounded spin models like 
the p-SOS model and finite-alphabet clock models. 
We prove the existence of families of distinct homogeneous tree-indexed Markov chain Gibbs states $\mu_A$ 
whose single-site marginals concentrate on a given finite subset $A\subset S$ of spin values, 
under a strong coupling condition for the interaction, 
depending only on the cardinality $\vert A \vert$ of $A$. 
The existence of such states is a new and robust phenomenon which is of 
particular relevance for infinite spin models. 
These states are not convex combinations of each other, and 
in particular the states with $\vert A \vert \geq 2$ can not be decomposed into homogeneous Markov-chain Gibbs states with a single-valued concentration center. 
As a further application of the method we obtain moreover the existence
of new types of $\Z$-valued gradient Gibbs states, whose single-site marginals 
do not localize, but whose correlation structure depends on the finite set $A$.  }
\end{abstract}
\textbf{MSC2020 subject classifications:} 82B26 (primary);
60K35 (secondary) 

\newpage 

\tableofcontents

\section{Introduction}	 

\subsubsection*{Gradient models on lattices and trees}	 

Statistical mechanics models for $\Z$- or $\R$-valued height variables with gradient interactions 
have been studied in a number of variations. For homogeneous models with different types of base spaces (such as lattices and trees)  
and different interaction potentials, see  
\cite{CoDeuMu09},\cite{DeuGiIo00},\cite{FS97},\cite{KoL14},\cite{Sh05},\cite{LaOt21},\cite{Ve06},\cite{CaLuEyMaSlTo14}.   
For disordered models with quenched randomness in the interaction, see \cite{DaHaPe21},\cite{BK94},\cite{CKu15},\cite{EnKu08},\cite{CoKuLN22}. 

In this paper we study $S$-valued gradient models on $d$-regular trees, whose interactions are defined by
transfer operators given by an even function $Q:S \rightarrow (0,\infty)$. Here, $S$ is assumed to be a countable Abelian group which we think of as the local state space of the system. In particular, 
if $S \subset \mathbb{R}$ is infinite, then the local state space can be viewed as the height-dimension of the system.
In concrete applications in statistical mechanics, we often encounter the case $S=\Z$ and the transfer operator 
is given by $Q(i)=\exp(-\beta U(|i|))$, 
where $U:S\rightarrow\R$ is a potential function prescribing the energetic cost of a spin configuration 
to make an increment of size $|i|$ along an edge of the tree, or more generally an edge of 
the supporting graph. 

Important special cases for the choice 
of the potential are the p-SOS model with $U(|i|)=|i|^p$, for which exponents $p\in (0,\infty )$
are allowed.  The most popular choices are 
$p=1$ which corresponds to the classical SOS model, and $p=2$ which gives the discrete Gaussian (see \cite{BaPaRo22,BaPaRo22a} for an analysis on the lattice.)
In our present approach besides positivity and evenness of $Q$, however, we will make no assumption on monotonicity or convexity in the interaction, and treat the function $Q$ as an infinite-dimensional parameter of the model. 
In contrast to this, \cite{LaTo22} provides an extensive description of gradient Gibbs measures for certain classes of gradient models on the tree whose underlying potential function is strictly convex.

The main interest of the study 
is in the construction and description of infinite-volume Gibbs measures (GM)
given by the DLR-consistency equation, as well as in gradient Gibbs measures (GGM) in the case of non-compact 
local spin space $S$ (for general background see \cite{Ge11},\cite{Sh05},\cite{FS97}).  GGMs are relevant as 
generalizations of the concept of GMs since they are suited to describe infinite-volume states 
which \textit{do not} localize in any bounded region $A\subset S$. 
In our present work we will only consider homogeneous (tree-automorphism invariant) 
measures. For some results on non-homogeneous measures on trees with homogeneous interactions we refer to \cite{GaRuSh12},\cite{GaMaRuSh20},\cite{AkRoTe11},\cite{HeKu21},\cite{CoKuLN22}. 

\subsubsection*{Main new result: Localization for \texorpdfstring{$S$}{S}-valued model on arbitrary finite localization sets \texorpdfstring{$A$}{A} } 
In a previous paper \cite{HeKu21a} two of us considered the case $S=\Z$ and proved the existence of localized Gibbs measures
under a strong coupling condition formulated for $Q$, namely 
boundedness of $(d+1)/2$-norm and small deviation of $Q$ from $1_{\{0\}}$ in terms of the $d+1$-norm. 
We showed that there are 
homogeneous states $\mu_i$ whose single-site
marginals are concentrated around \textit{single fixed heights} $i \in \Z$. 

In the present paper we will extend this type of result 
to the case of arbitrary finite localization sets $A\subset S$, 
for the height variables, under appropriate strong coupling conditions. 
The strong-coupling states with non-singleton concentration sets 
are of a new type in the setup of unbounded spin models, and 
to our best knowledge have not been discussed before. 
To appreciate the result it is important to note that these new tree-indexed Markov chain Gibbs measures $\mu_A$ 
constructed in our present work are not convex combinations of each other, and in particular 
not of the spatially homogeneous measures $\mu_i$ with single-height concentration which were constructed in \cite{HeKu21a}.

The existence 
result of our new Theorem \ref{main1-thm} holds under under an $N$-dependent 
strong-coupling condition on $Q$ which provides existence of 
measures $\mu_A$ which concentrate on
localization sets $A$ of size $|A|\leq N$. 
It is particularly remarkable that under this condition  
the localization sets in these families can 
be arbitrarily spread out.    

This existence result may look surprising but can be made plausible by seeing it as an infinite-dimensional 
generalization of a simpler phenomenon which is known to appear in the $q$-state Potts model on the tree. 
The homogeneous Markov chain states of the Potts model can described via explicit 
computations \cite{KuRo14}, due to the full 
invariance of the interaction under permutation in local spin-space. 

Our proof in the general infinite-dimensional case 
is based on the boundary law description of 
Gibbs measures going back to Zachary \cite{Z83}. 
In the present case of gradient interactions on $S$ 
this leads to a non-linear fixed point equation in the space $\ell^{d+1}(S)$ of $d+1$-summable
functions $u: S \rightarrow \R$.
In general explicit solutions 
are out of question, and for our proof we will develop a fixed point method,  
adapted with a view to the type of $A$-dependent states we are hoping to find, see Section \ref{existence-sec}.

Our approach to study the infinite-dimensional fixed point problem is 
to break the problem into two parts: on the given finite concentration set $A$ where we expect 
to find the large components, 
and a conditional problem away from it where we expect to find the small components.  
For the latter we devise a suitable (conditional) 
map on sequence space which we show to be a contraction, 
for the former we employ the Brouwer fixed point theorem, see \eqref{eq: MapG}. 
This leads us to quite explicit quantitative thresholds for the system parameters of given 
models for which we can prove existence of $A$-concentrated states, see Proposition \ref{existence}.
On the level of system parameters, the strong coupling condition on $Q=\exp(-\beta U)$ 
translates into the fact that the parameter $\beta$ - which as usual should
be interpreted as the inverse temperature - should be large enough, see Section \ref{examples-sec}. 
For a discussion of uniqueness, see the Remarks \ref{uniqueness1} and \ref{uniqueness1b}. 

\subsubsection*{Harvesting new families of delocalized gradient Gibbs measures (GGM)} 
In the case $S=\Z$, next to proper Gibbs measures, another class of consistent measures, namely 
the gradient Gibbs measures have received much attention, see \cite{FS97,KoL14,Sh05,HKLR19}. 
Gradient Gibbs measures (as opposed to Gibbs measures) 
are measures which are only defined 
on $\Z^V/\Z$, which is the space of infinite-volume height configurations modulo a joint height shift 
(as opposed to the state space of absolute heights $\Z^V$ itself). Their defining property is the validity of the 
DLR consistency equation, but read only modulo joint height-shift, for details see Subsection \ref{subsec: DefGGM}. 
As a consequence of the first part of our work we also 
obtain new families of gradient Gibbs measures $\nu^{q}_{A}$, where $q\geq 2$ is 
an integer and $A\subset \Z^q$. 
They have the delocalization property (see Theorem \ref{thm: deloc}), and hence do not stem from homogeneous 
Markov chain GMs.  Therefore they are completely 
different in character from the localized GMs. 

We construct these states $\nu^{q}_{A}$ as follows. 
The idea is to relate to the Gibbs measures $\mu^q_A$ in an associated $q$-state 
clock model on $\Z_q$ 
with an interaction $Q^q$ built from the original interaction $Q$ on $\Z^V$
a gradient Gibbs measure $\nu$ on $\Z^V/\Z$. This is done  
via an edge-wise resampling procedure, see Subsection \ref{subsec: ConstructionOfGGM}.
The concentration properties of the clock-measures 
$\mu^q_A$ we constructed in our first main Theorem \ref{main1-thm}, then carries over 
to an interesting $A$-dependent correlation structure for the gradient measures $\nu^{q}_{A}$, 
see Corollary \ref{thm: PeriodicGGM} and the discussion below.
The remainder of the paper is organized as follows: In Section \ref{sec: Definitions} we define our models. Section \ref{results-sec} then contains our main results regarding Gibbs measures for arbitrary finite concentration sets $A \subset S$. Section \ref{sec GGM} discusses existence of delocalized gradient states with $A$-dependent correlation structure. Finally, Section \ref{sec: Proofs} contains the proofs.  

\newpage

\section{Definitions}\label{sec: Definitions} 

In this section, we review some definitions and known facts which are necessary in order to formulate our main result.

\subsection{Spin configurations on the Cayley tree}

Let $\Gamma^d=(V,E)$ denote the \textit{d-regular tree} or \textit{Cayley tree} of order $d \geq 2$, where $V$ is the countably infinite set of vertices and $E \subset V \times V$ is the set of (unoriented) edges. The term \textit{d-regular tree} means that the graph $\Gamma^d$ is connected without cycles and each vertex $x \in V$ has exactly $d+1$ nearest neighbors, i.e., vertices which are connected to $x$ by an edge.

A \textit{path} connecting two vertices $x,y \in V$ is an ordered list of $n$ edges 
\[
\{x,x_1\},\{x_1,x_2\},\ldots,\{x_{n-1},y\}
\] 
where any two consecutive edges share a common vertex. The length of the unique shortest path from $x$ to $y$ defines the \textit{graph distance} $\text{d}(x,y)$. 

Besides the set of unoriented edges $E$, we also consider the set $\vec{E}$ of oriented edges, which consists of the ordered pairs $(x,y)$ of vertices such that $\{x,y\} \in E$.

For any subset $\Lambda \subset V$, we denote by $\Lambda^c$ the complement of $\Lambda$ in $V$ and by
\[\partial \Lambda:= \{x \in \Lambda^c \mid \text{d}(x,y)=1 \text{ for some } y \in \Lambda\}\]
the \textit{outer boundary} of $\Lambda$. By $\Lambda \Subset V$ we indicate that $\Lambda$ is a finite subset of $V$. 

We set 
\[
E_\Lambda:=\{\{x,y\} \in E \mid x,y \in \Lambda\}, \qquad \vec{E}_\Lambda:=\{(x,y) \in \vec{E} \mid x,y \in \Lambda\}.
\]
and note that the pair $(\Lambda,E_\Lambda)$ is a subgraph of $\Gamma^d$, which is a subtree if and only if it is connected. 

Let $(S,+)$ be a countable Abelian group, which we think of as the \textit{local state space} of our system. Important particular cases are given by the lattices $S=\Z^k$, $k\in \N$, and by the finite cyclic groups $S=\Z_q$, $q\in \N$. We see $S$ as a discrete group and endow it with the measurable structure given by the whole power set $\mathcal{P}(S)$. 

By the symbol $\ell^p(S)$, $1\leq p \leq \infty$, we denote the space of $p$-summable real valued functions on $S$, which is a Banach space with the norm
\[
\|u\|_p := \left( \sum_{i\in S} |u(i)|^p \right)^{\frac{1}{p}} \qquad \mbox{for } 1\leq p < \infty, \qquad \|u\|_{\infty} := \sup_{i\in S} |u(i)|.
\] 
We recall that
\[
\ell^p(S) \subset \ell^q(S) \qquad \mbox{and} \qquad \|u\|_q \leq \|u\|_p \qquad \mbox{if} \quad 1\leq p \leq q \leq \infty.
\]
When the group $S$ is finite, the spaces $\ell^p(S)$ are of course independent of $p$, as they all coincide with $\R^S$, but the $p$-norms on them are different.  Convolution on $S$ is denoted by
\[
(u * v) (i) := \sum_{j\in S} u(i-j) v(j).
\]
A \textit{spin configuration} $\omega=(\omega_x)_{x \in V}$ is a map from the set of vertices $V$ to the local state space $S$, and the set of all spin configurations is denoted by $\Omega:=S^V$. For any subset $\Lambda \subset V$ and any $\omega\in \Omega$, we set $\Omega_\Lambda:=S^\Lambda$ and denote by $\omega_{\Lambda}\in \Omega_\Lambda$ the restriction of $\omega$ to $\Lambda$. 

We endow each $\Omega_\Lambda$ with the product $\sigma$-algebra $\mathcal{F}_\Lambda$ generated by the \textit{spin projections} $\sigma_y: \Omega_\Lambda \rightarrow S, \ \sigma_y(\omega)=\omega_y$, where $y \in \Lambda$, and denote by $\mathcal{F}:=\mathcal{F}_V$ the product $\sigma$-algebra on $\Omega$. 

The set of all probability measures on the space $(\Omega,\mathcal{F})$ is denoted by \\ $\mathcal{M}_1(\Omega,\mathcal{F})$.
 We call a probability measure $\mu\in \mathcal{M}_1(\Omega, \mathcal{F})$ \textit{(spatially) homogeneous} if it is invariant under all automorphisms $\varphi: V \rightarrow V$ of the tree, i.e., $\mu = \mu \circ \varphi_*^{-1}$, where $\varphi_*: \mathcal{F} \rightarrow \mathcal{F}$ is the map $\varphi_*(A):=\{\omega \circ \varphi^{-1} \mid \omega \in A\}$.

Given a spatially homogeneous probability measure $\mu$, we denote by
\[
\pi_{\mu}(i):= \mu ( \sigma_x = i ), \qquad P_{\mu}(i,j) := \mu (\sigma_x = j \mid \sigma_y = i), \qquad \forall i,j\in S,
\]
the \textit{single state marginal} and the \textit{transition matrix} induced by $\mu$. Here, $x\in V$ is any vertex and $\{x,y\}\in E$ is any edge, but the above objects do not depend on these choices because the measure $\mu$ is assumed to be  spatially homogeneous.

\subsection{Tree-indexed Markov chains}

The notion of a tree-indexed Markov chain as given in Chapter 12 of \cite{Ge11} is based on the definition of the past of an oriented edge.
Given any vertex $v \in V$ we write
\[
{}^v\vec{E}:=\{(x,y) \in \vec{E} \mid \text{d}(y,v)=\text{d}(x,v)+1\}
\]
for the oriented edges pointing away from $v$. The \textit{past of an oriented edge} $(x,y) \in \vec{E}$ is then defined by 
\[ 
]-\infty,xy]:=\{v \in V \mid (x,y) \in {}^v\vec{E} \}.
\]

In other words, it consists of those vertices $v \in V$ for which the shortest path from $v$ to $y$ contains $x$.
A probability measure $\mu$ on $(\Omega,\mathcal{F})$ is then called a \textit{tree-indexed Markov chain} (or simply a \textit{Markov chain}) if for all oriented edges $(x,y) \in \vec{E}$ and all $i \in S$ we have
\[
\mu(\sigma_y=i \mid \mathcal{F}_{]-\infty,xy]})=\mu(\sigma_y=i \mid \mathcal{F}_{x}) \quad \mu-\text{a.s.}
\]

\subsection{Transfer operators and Gibbs measures}

In this paper, by a \textit{transfer operator} on $(\Gamma^d,S)$ we mean a function 
\[
Q: S \rightarrow (0,+\infty)
\]
which is symmetric (i.e., $Q(-i) = Q(i)$ for every $i\in S$) and belongs to $\ell^{\frac{d+1}{2}}(S)$. A more precise name for such an object would be \textit{spatially homogeneous positive symmetric transfer operator}. Often, transfer operators are given in terms of a suitable even interaction function $U : S \rightarrow [0,+\infty)$ as
\[
Q(i) = e^{-\beta U(i)},
\]
where $\beta>0$ should be interpreted as the inverse of a temperature. 

A transfer operator $Q$ induces the \textit{Markovian gradient specification}
\[
\gamma = \{\gamma_{\Lambda}: \mathcal{F} \times \Omega \rightarrow [0,1] \}_{\Lambda  \Subset V}
\]
by the assignment
\begin{equation}\label{eq: GibbsSpecification}
\gamma_\Lambda(\sigma_\Lambda= \tilde{\omega} \mid \omega)=\frac{1}{Z_\Lambda(\omega_{\partial \Lambda})}\left(\prod_{\{x,y\} \in E_\Lambda}Q(\tilde{\omega}_x-\tilde{\omega}_y) \right) \, \prod_{\genfrac{}{}{0pt}{}{\{x,y\} \in E}{ x \in \Lambda, y \in \partial \Lambda}}Q(\tilde{\omega}_x-\omega_y), 
\end{equation}
for every $\Lambda\Subset V$, $\tilde\omega\in \Omega_{\Lambda}$ and $\omega\in \Omega$. Here, the \textit{partition function} $Z_{\Lambda}$ gives for every $\omega\in \Omega$ the normalization constant $Z_{\Lambda}(\omega) = Z_{\Lambda}(\omega_{\partial\Lambda})$ turning $\gamma_{\Lambda}(\cdot \mid \omega)$ into a probability measure on $(\Omega,\mathcal{F})$. The assumptions $Q>0$ and $Q\in \ell^{\frac{d+1}{2}}(S)$ guarantee that such a partition function does exist. See Lemma 1 in \cite{HeKu21a} (here the case $S=\Z^k$ is considered, but the proof immediately generalizes to the case of an arbitrary countable group $S$). The quantities $Q(\tilde{\omega}_x-\tilde{\omega}_y)$ and $Q(\tilde{\omega}_x-\omega_y)$ are well defined for $\{x,y\}\in E$ because $Q$ is assumed to be symmetric.

\begin{rk}\label{rk: EquivOfPot}
{\rm Note that if $\tilde{Q}=c\,Q$ for some $c >0$, then the Markovian gradient specifications which are induced by $Q$ and $\tilde{Q}$ coincide. We shall often find it useful to normalize $Q$ by requiring $Q(0)=1$.}
\end{rk}

A \textit{Gibbs measure} for a specification $\gamma$ (a transfer operator $Q$, respectively) is by definition a probability measure $\mu$ on $(\Omega,\mathcal{F})$, such that for all $\Lambda \Subset V$ and all $A \in \mathcal{F}$ the \textit{Dobrushin-Lanford-Ruelle (DLR) equation} 
\[
\mu(A)= \int \gamma_\Lambda(A \mid \omega) \mu(\text{d} \omega) 
\]
holds true.

We denote the (possibly empty) convex set of Gibbs measures on $(\Omega, \mathcal{F})$ for a specification $\gamma$ by $\mathcal{G}(\gamma)$. If $\mathcal{G}(\gamma)$ is not empty, then each of its elements is the convex combination of  extremal elements of $\mathcal{G}(\gamma)$ (see eg. Thm.  7.26 in \cite{Ge11}). On the tree, each such extremal Gibbs measure for a Markovian specification is a tree-indexed Markov chain (eg. Thm 12.6 in \cite{Ge11} for this statement in the case in which $S$ is finite; the proof generalizes to countable local state spaces). Writing $\text{ex} \,C$ for the set of extremal points of a convex set $C$ and $\mathcal{MG}(\gamma)$ for the set of Gibbs measures for $\gamma$ which are Markov chains, the above statement reads
\begin{equation}
\text{ex}\, \mathcal{G}(\gamma) \subset \mathcal{MG}(\gamma) \subset \mathcal{G}(\gamma). 
\end{equation}	

\begin{rk}\label{distance} 
{\rm
Assume that the transfer operator $Q$ on the discrete Abelian group $S$ is normalized by $Q(0)=1$ and satisfies $Q(i)<1$ for every $i\in S\setminus \{0\}$. Then $Q$ induces the translation invariant ``distance function''
\[
\mathrm{dist}_Q(i,j) := - \log Q(i-j), \qquad \forall i,j\in S,
\]
where the quotes refer to the fact that the function $\mathrm{dist}_Q(i,j)$ is symmetric, non-negative, zero if and only if $i=j$, but in general does not satisfy the triangle inequality. It is a genuine distance function if $Q$ satisfies the $\log$-superadditivity condition $Q(i+j)\geq Q(i)Q(j)$. If we write $Q=e^{-U}$, where $U$ is a non-negative even function on $S$ vanishing  only at zero, then this is equivalent to the subadditivity of $U$. This condition is indeed fulfilled by many models, including the SOS-model and the Log-model which are analysed in Section \ref{examples-sec} below, while it is not satisfied, e.g., for the p-SOS model with potential $U(i) = \vert i \vert^p$ when $p>1$.}
\end{rk}

\section{Main result}\label{results-sec}

In this section, we present the main result of this paper regarding the existence of Markov-chain Gibbs measures on the regular Cayley $d$-tree with a countable Abelian group $S$ as local state space. The Gibbs measures we find localize on an arbitrary finite subset of $S$. We also discuss some of its immediate implications and applications.

\subsection{Existence of Gibbs measures localizing on finite sets}
\label{main1-sec}

In order to formulate the main existence result, we need to introduce some functions of the order $d$ of the Cayley tree and of the cardinality $n$ of the subsets of $S$ on which our Gibbs measure will localize. We denote by $\rho=\rho(d,n)$ the unique positive number such that
\[
(d-1)\rho^{d+1} + dn \rho^{d-1} - n = 0,
\]
and we set
\[
\eta(d,n) :=  \frac{\rho-\rho^d}{(\rho^{d+1} + n)^{\frac{d}{d+1}}}.
\]
Note that $\rho$ belongs to the interval $(0,1)$, so $\eta(d,n)$ is a positive number. We actually have the bounds $\rho(d,n) < d^{-\frac{1}{d-1}}$ and
\begin{equation}
\label{bb0}
d^{-\frac{1}{d-1}} \left( 1 - \frac{1}{d} \right) (n+1)^{- \frac{d}{d+1}} \leq \eta(d,n) \leq d^{-\frac{1}{d-1}} \left( 1 - \frac{1}{d} \right) n^{- \frac{d}{d+1}}< 1,
\end{equation}
as shown in Lemma \ref{asymptotics} in Section \ref{existence-sec} below. We can now state the main existence result of the paper:

\begin{thm}
\label{main1-thm}
Let $d\geq 2$ and $N\geq 1$ be integers. Assume that the transfer operator $Q\in \ell^{\frac{d+1}{2}}(S)$ is normalized by $Q(0)=1$ and satisfies the condition
\begin{equation}
\label{mainass}
\|Q-\mathbbm{1}_{\{0\}}\|_{\frac{d+1}{2}} \leq \eta(d,N).
\end{equation}
Then for every $A\subset S$ with $1\leq |A|\leq N$ the Markovian gradient specification which is induced by $Q$ on the regular $d$-tree with local state space $S$ admits a spatially homogeneous Markov-chain Gibbs measure $\mu$ such that, denoting by $\pi_{\mu}$ the single site marginals of $\mu$ and by
\[
\Delta(i):= P_{\mu}(i,i), \qquad i\in S,
\]
the diagonal elements of the transition matrix, we have:
\begin{equation}
\label{main2}
\|\pi_{\mu}|_{A^c}\|_1 < \theta \, \min_A \pi_{\mu}, \qquad \mbox{with } \theta := \bigl( d^{\frac{1}{d-1}}  \rho(d,|A|) \bigr)^{d+1} \in (0,1),
\end{equation}
\begin{equation}
\label{main1}
 \|\Delta|_{A^c}\|_{\frac{d+1}{d-1}} < \rho(d,|A|)^{d-1} <  \frac{1}{d} < \Delta|_A.
\end{equation}
Moreover, setting $\epsilon := \| Q-\mathbbm{1}_{\{0\}}\|_{\frac{d+1}{2}}$ and $n:=|A|$, the following estimates hold:
\begin{enumerate}[(i)]
\item $\|\Delta|_{A^c}\|_{\frac{d+1}{d-1}} \leq c_1\, \epsilon^{d-1}$, where $c_1 = c_1(d,n) := \frac{\rho(d,n)^{d-1}}{\eta(d,n)^{d-1}}$.
\item $\Delta|_A > 1 - c_2\, \epsilon$, where $c_2 = c_2(d,n) := (d^{\frac{d}{d-1}} - d)(1+n)^{\frac{d}{d+1}}$.
\item $\|\pi_{\mu}|_{A^c}\|_1 \leq c_3\, \epsilon^{d+1}$, where $c_3 = c_3(d,n) := \frac{d^{\frac{d+1}{d-1}} \rho(d,n)^{d+1}}{n \, \eta(d,n)^{d+1}} $.
\item $(1-c_5 \, \epsilon) \frac{1}{|A|} \leq \pi_{\mu}|_A \leq (1-c_4 \, \epsilon)^{-1} \frac{1}{|A|}$,
where $c_4 = c_4(d,n) := \frac{d+1}{d-1} c_2(d,n)$ and $c_5 = c_5(d,n) := c_4(d,n) + \frac{\rho(d,n)^{d+1}}{n \, \eta(d,n)^{d+1}}$.
\item $\pi_{\mu}(i) \geq \frac{1}{|A|} (1-c_6 \,\epsilon) \left( \sum_{j\in A} Q(i-j) \right)^{d+1}$ for every $i\in A^c$, where $c_6=c_6(d,n):=d\, c_4(d,n) +  \frac{\rho(d,n)^{d+1}}{n \, \eta(d,n)^{d+1}} $.
\end{enumerate}
\end{thm}

Some comments are in order. First of all notice that the assumption \eqref{mainass} involves the cardinality of $A$ but neither the Abelian group $S$ nor the way in which $A$ sits in $S$. Thanks to \eqref{bb0},  \eqref{mainass} implies that $Q(i)<1$ for every $i\in S\setminus \{0\}$.

Condition \eqref{main2} tells us that the spin values in $A$ are preferred by the Gibbs measure $\mu$: the probability that a given vertex is not in $A$ is smaller than the probability that it is in the least likely of the spin values of $A$.
The bounds (iii) and (iv) control how the probability distribution $\pi_{\mu}$ giving us the single state marginals converges to the equidistribution on $A$ as the $\frac{d+1}{2}$-norm of $Q - \mathbbm{1}_{\{0\}}$ tends to zero.

Condition \eqref{main1} and its asymptotic $Q$-dependent refinements (i) and (ii) tell us that $A$ is the ``lazy'' set of the Gibbs measure $\mu$. Indeed, in the case $d=2$ \eqref{main1} says that if a vertex is in a state $i$ belonging to the set $A$, then its neighbouring vertices will prefer to remain in $i$ with probability larger than $\frac{1}{2}$; otherwise, they will prefer to change their state with probability larger than $\frac{1}{2}$. When the order $d$ increases, the probability threshold has the smaller value $\frac{1}{d}$: if the number $d+1$ of  vertices that influence the state at a given vertex gets larger, then a change of state becomes more probable for states in $A^c$, but possibly also for those in $A$.

The asymptotic $Q$-dependent bounds (i) and (ii) moreover quantify how the ``laziness'' of $A$ gets stronger and stronger as the $\frac{d+1}{2}$-norm of $Q - \mathbbm{1}_{\{0\}}$ tends to zero. Both the probability of changing state if $i$ is not in $A$ and the probability of keeping the same state for $i$ in $A$ tend to one. 

Given any state $i$, the probability to go from $i$ to some state in $A^c$ along some edge also tends to zero as $\epsilon:= \|Q - \mathbbm{1}_{\{0\}} \|_{\frac{d+1}{2}}$ tends to zero. This is clear for states $i$ in $A$. For states $i$ in $A^c$ it can be shown as follows. Thanks to the formula
\begin{equation}
\label{transitionfor0}
P_{\mu}(i,j) = \frac{\pi_{\mu}(j)^{\frac{d}{d+1}} Q(i-j)}{\bigl( Q * \pi_{\mu}^{\frac{d}{d+1}} \bigr) (i)},
\end{equation}
which is discussed in Remark \ref{transitionfor} below, the H\"older inequality and the bound (iii) imply
\[
\begin{split}
P_{\mu}(i,A^c) :&= \sum_{j\in A^c} P_{\mu}(i,j) = \frac{1}{(Q * \pi_{\mu}^{\frac{d}{d+1}})(i)} \sum_{j\in A^c} \pi_{\mu}^{\frac{d}{d+1}}(j) Q(i-j) \\ &\leq \frac{1}{(Q * \pi_{\mu}^{\frac{d}{d+1}})(i)}  \|\pi_{\mu}|_{A^c}\|_1^{\frac{d}{d+1}} \|Q\|_{d+1} \leq  \frac{c_7}{(Q * \pi_{\mu}^{\frac{d}{d+1}})(i)}  \|Q\|_{d+1} \, \epsilon^d,
\end{split}
\]
where $c_7:= c_3^{\frac{d}{d+1}}$. This already implies our claim for any fixed $i\in A^c$. Together with the lower bound (v), we obtain the further estimate
\begin{equation}
\label{prov}
P_{\mu}(i,A^c) \leq c_7 (1-c_6 \, \epsilon)^{- \frac{d}{d+1}} |A|^{\frac{d}{d+1}} \|Q\|_{d+1} \Bigl( \sum_{j\in A} Q(i-j) \Bigr)^{-d} \epsilon^d.
\end{equation}
Recalling that $\mathrm{dist}_Q$ denotes the ``distance function'' discussed in Remark \ref{distance}, we find
\[
\begin{split}
\Bigl( \sum_{j\in A} Q(i-j) \Bigr)^{-d} &\leq \bigl( \max_{j\in A} Q(i-j) \bigr)^{-d} = \min_{j\in A} Q(i-j)^{-d} = \min_{j\in A} e^{-d \, \log Q(i-j)} \\ &= 
\min_{j\in A} e^{d \, \mathrm{dist}_Q(i,j)} =  e^{d \, \min_{j\in A} \mathrm{dist}_Q (i,j) } = e^{d\, \mathrm{dist}_Q(i,A)}.
\end{split}
\]
Therefore, \eqref{prov} implies the estimate
\[
P_{\mu}(i,A^c) \leq c_7 (1-c_6 \, \epsilon)^{- \frac{d}{d+1}} |A|^{\frac{d}{d+1}} \|Q\|_{d+1} e^{d\, \mathrm{dist}_Q(i,A)} \, \epsilon^d,
\]
which tells us that as $\epsilon$ tends to zero $P_{\mu}(i,A^c)$ converges to zero uniformly for $i$ in any subset of $A^c$ whose elements have uniformly bounded distance from $A$. 

A graphical illustration of the above discussion is given in Figure \ref{fig: lazyset}.

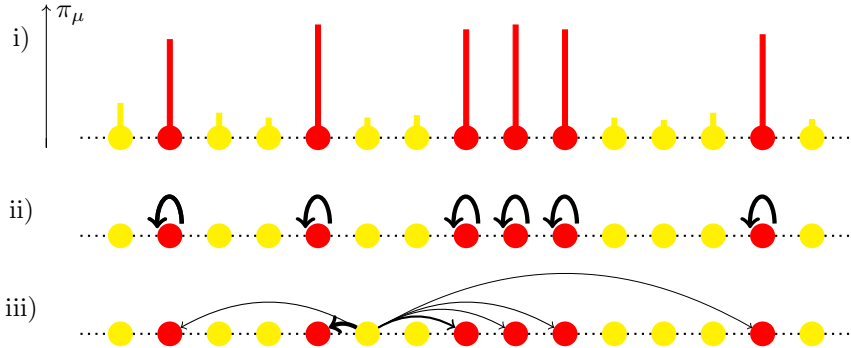
\begin{figure}[h]
	\centering
	\begin{tikzpicture}[scale=.65]
\node (ya1) at (-8.5,3.8) {};
\node (ya2) at (-8.5, 6.5) {};
\node (yl) at (-8,6.5) {$\pi_{\mu}$};	
\node at (-9,6) {i)};	
\node (g0) at (-8,4) {};
\node (g1) at (8,4) {};		
\node (c1) at (-7,4) [circle, draw=yellow,thick,fill=yellow,minimum size = 3mm, inner sep = 0]{};		
\node (c2) at (-6,4) [circle, draw=red,thick,fill=red,minimum size = 3mm, inner sep = 0]{};
\node (c3) at (-5,4) [circle, draw=yellow,thick,fill=yellow,minimum size = 3mm, inner sep = 0]{};
\node (c4) at (-4,4) [circle, draw=yellow,thick,fill=yellow,minimum size = 3mm, inner sep = 0]{};
\node (c5) at (-3,4) [circle, draw=red,thick,fill=red,minimum size = 3mm, inner sep = 0]{};
\node (c6) at (-2,4) [circle, draw=yellow,thick,fill=yellow,minimum size = 3mm, inner sep = 0]{};
\node (c7) at (-1,4) [circle, draw=yellow,thick,fill=yellow,minimum size = 3mm, inner sep = 0]{};
\node (c8) at (0,4) [circle, draw=red,thick,fill=red,minimum size = 3mm, inner sep = 0]{};
\node (c9) at (1,4) [circle, draw=red,thick,fill=red,minimum size = 3mm, inner sep = 0]{};
\node (c10) at (2,4) [circle, draw=red,thick,fill=red,minimum size = 3mm, inner sep = 0]{};
\node (c11) at (3,4) [circle, draw=yellow,thick,fill=yellow,minimum size = 3mm, inner sep = 0]{};
\node (c12) at (4,4) [circle, draw=yellow,thick,fill=yellow,minimum size = 3mm, inner sep = 0]{};
\node (c13) at (5,4) [circle, draw=yellow,thick,fill=yellow,minimum size = 3mm, inner sep = 0]{};
\node (c14) at (6,4) [circle, draw=red,thick,fill=red,minimum size = 3mm, inner sep = 0]{};
\node (c15) at (7,4) [circle, draw=yellow,thick,fill=yellow,minimum size = 3mm, inner sep = 0]{};
\draw[draw=yellow,fill=yellow] (-7.05,4.2) rectangle ++(0.1,0.5);
\draw[draw=red,fill=red] (-6.05,4.2) rectangle ++(0.1,1.8);
\draw[draw=yellow,fill=yellow] (-5.05,4.2) rectangle ++(0.1,0.3);
\draw[draw=yellow,fill=yellow] (-4.05,4.2) rectangle ++(0.1,0.2);
\draw[draw=red,fill=red] (-3.05,4.2) rectangle ++(0.1,2.1);
\draw[draw=yellow,fill=yellow] (-2.05,4.2) rectangle ++(0.1,0.2);
\draw[draw=yellow,fill=yellow] (-1.05,4.2) rectangle ++(0.1,0.25);
\draw[draw=red,fill=red] (-0.05,4.2) rectangle ++(0.1,2);
\draw[draw=red,fill=red] (0.95,4.2) rectangle ++(0.1,2.1);
\draw[draw=red,fill=red] (1.95,4.2) rectangle ++(0.1,2);
\draw[draw=yellow,fill=yellow] (2.95,4.2) rectangle ++(0.1,0.2);
\draw[draw=yellow,fill=yellow] (3.95,4.2) rectangle ++(0.1,0.15);
\draw[draw=yellow,fill=yellow] (4.95,4.2) rectangle ++(0.1,0.3);
\draw[draw=red,fill=red] (5.95,4.2) rectangle ++(0.1,1.9);
\draw[draw=yellow,fill=yellow] (6.95,4.2) rectangle ++(0.1,0.17);
\foreach \from/\to in 
{g0/c1,c1/c2,c2/c3,c3/c4,c4/c5,c5/c6,c6/c7,c7/c8,c8/c9,c9/c10,c10/c11,c11/c12,c12/c13,c13/c14,c14/c15,c15/g1}
\draw[dotted,thick] (\from) to (\to);
\draw[->] (ya1) |- ++(0,0) |- (ya2);
\node (f0) at (-8,2) {};
\node (f1) at (8,2) {};	
\node (yb) at (-9,2.5) {ii)};	
\node (b1) at (-7,2) [circle, draw=yellow,thick,fill=yellow,minimum size = 3mm, inner sep = 0]{};		
\node (b2) at (-6,2) [circle, draw=red,thick,fill=red,minimum size = 3mm, inner sep = 0]{};
\node (b3) at (-5,2) [circle, draw=yellow,thick,fill=yellow,minimum size = 3mm, inner sep = 0]{};
\node (b4) at (-4,2) [circle, draw=yellow,thick,fill=yellow,minimum size = 3mm, inner sep = 0]{};
\node (b5) at (-3,2) [circle, draw=red,thick,fill=red,minimum size = 3mm, inner sep = 0]{};
\node (b6) at (-2,2) [circle, draw=yellow,thick,fill=yellow,minimum size = 3mm, inner sep = 0]{};
\node (b7) at (-1,2) [circle, draw=yellow,thick,fill=yellow,minimum size = 3mm, inner sep = 0]{};
\node (b8) at (0,2) [circle, draw=red,thick,fill=red,minimum size = 3mm, inner sep = 0]{};
\node (b9) at (1,2) [circle, draw=red,thick,fill=red,minimum size = 3mm, inner sep = 0]{};
\node (b10) at (2,2) [circle, draw=red,thick,fill=red,minimum size = 3mm, inner sep = 0]{};
\node (b11) at (3,2) [circle, draw=yellow,thick,fill=yellow,minimum size = 3mm, inner sep = 0]{};
\node (b12) at (4,2) [circle, draw=yellow,thick,fill=yellow,minimum size = 3mm, inner sep = 0]{};
\node (b13) at (5,2) [circle, draw=yellow,thick,fill=yellow,minimum size = 3mm, inner sep = 0]{};
\node (b14) at (6,2) [circle, draw=red,thick,fill=red,minimum size = 3mm, inner sep = 0]{};
\node (b15) at (7,2) [circle, draw=yellow,thick,fill=yellow,minimum size = 3mm, inner sep = 0]{};
\foreach \from/\to in 
{f0/b1,b1/b2,b2/b3,b3/b4,b4/b5,b5/b6,b6/b7,b7/b8,b8/b9,b9/b10,b10/b11,b11/b12,b12/b13,b13/b14,b14/b15,b15/f1}
\draw[dotted,thick] (\from) to (\to);
\node (a2e) at (-5.95,2.25) {};	
\node (a5e) at (-2.95,2.25) {};	
\node (a8e) at (0.05,2.25) {};	
\node (a9e) at (1.05,2.25) {};	
\node (a10e) at (2.05,2.25) {};
\node (a14e) at (6.05,2.25) {};		
\draw[ultra thick,->] (a2e.0) arc (0:190:2.5mm and 5.5mm);
\draw[ultra thick,->] (a2e.0) arc (0:190:2.5mm and 5.5mm);
\draw[ultra thick,->] (a5e.0) arc (0:190:2.5mm and 5.5mm);
\draw[ultra thick,->] (a8e.0) arc (0:190:2.5mm and 5.5mm);
\draw[ultra thick,->] (a9e.0) arc (0:190:2.5mm and 5.5mm);
\draw[ultra thick,->] (a10e.0) arc (0:190:2.5mm and 5.5mm);
\draw[ultra thick,->] (a14e.0) arc (0:190:2.5mm and 5.5mm);
\node (e0) at (-8,0) {};
\node (e1) at (8,0) {};		
\node at (-9,0.5) {iii)};
\node (a1) at (-7,0) [circle, draw=yellow,thick,fill=yellow,minimum size = 3mm, inner sep = 0]{};		
\node (a2) at (-6,0) [circle, draw=red,thick,fill=red,minimum size = 3mm, inner sep = 0]{};
\node (a3) at (-5,0) [circle, draw=yellow,thick,fill=yellow,minimum size = 3mm, inner sep = 0]{};
\node (a4) at (-4,0) [circle, draw=yellow,thick,fill=yellow,minimum size = 3mm, inner sep = 0]{};
\node (a5) at (-3,0) [circle, draw=red,thick,fill=red,minimum size = 3mm, inner sep = 0]{};
\node (a6) at (-2,0) [circle, draw=yellow,thick,fill=yellow,minimum size = 3mm, inner sep = 0]{};
\node (a7) at (-1,0) [circle, draw=yellow,thick,fill=yellow,minimum size = 3mm, inner sep = 0]{};
\node (a8) at (0,0) [circle, draw=red,thick,fill=red,minimum size = 3mm, inner sep = 0]{};
\node (a9) at (1,0) [circle, draw=red,thick,fill=red,minimum size = 3mm, inner sep = 0]{};
\node (a10) at (2,0) [circle, draw=red,thick,fill=red,minimum size = 3mm, inner sep = 0]{};
\node (a11) at (3,0) [circle, draw=yellow,thick,fill=yellow,minimum size = 3mm, inner sep = 0]{};
\node (a12) at (4,0) [circle, draw=yellow,thick,fill=yellow,minimum size = 3mm, inner sep = 0]{};
\node (a13) at (5,0) [circle, draw=yellow,thick,fill=yellow,minimum size = 3mm, inner sep = 0]{};
\node (a14) at (6,0) [circle, draw=red,thick,fill=red,minimum size = 3mm, inner sep = 0]{};
\node (a15) at (7,0) [circle, draw=yellow,thick,fill=yellow,minimum size = 3mm, inner sep = 0]{};
\foreach \from/\to in 
{e0/a1,a1/a2,a2/a3,a3/a4,a4/a5,a5/a6,a6/a7,a7/a8,a8/a9,a9/a10,a10/a11,a11/a12,a12/a13,a13/a14,a14/a15,a15/e1}
\draw[dotted,thick] (\from) to (\to);
\path[->,thick] (a6) edge [bend left] node {} (a8);
\path[->,ultra thick] (a6) edge [bend right] node {} (a5);
\path[->, thin] (a6) edge [bend right] node {} (a2);
\path[->, thin] (a6) edge [bend left] node {} (a9);
\path[->] (a6) edge [bend left] node {} (a10);
\path[->,ultra thin] (a6) edge [bend left] node {} (a14);
\end{tikzpicture}
\caption{\label{fig: lazyset}\small The pictures show a part of the set $S=\Z$ with the bars in  picture i) marking the distribution of single-site marginals of the Gibbs measure $\mu$ from Theorem \ref{main1-thm}. Here, the red coloured circles and bars belong to $A \Subset S$. If the chain is in a state in $A$, then it prefers to stay in this state, see ii). On the other hand, being in a state which does not belong to $A$, the chain prefers jumps into states in $A$, with weights as indicated by the arrows in iii). Under suitable decaying conditions on $Q$, \eqref{transitionfor0} implies that shortest jumps are more likely.
}		
\end{figure}

\begin{rk}
\label{uniqueness0}
{\rm For $|A|=1$ and $S=\Z$, the existence of a Gibbs measure as in the above theorem has been proven in \cite{HeKu21a} under similar but not exactly equivalent assumptions on $Q$.}
\end{rk}

\begin{rk}{\rm ({\em Uniqueness}) If $|A|=1$ and $\|Q-\mathbbm{1}_{\{0\}}\|_{\frac{d+1}{2}} \leq \eta(d,1)$, then we can further show that the spatially homogeneous Markov-chain Gibbs measure $\mu$ satisfying \eqref{main1} is unique. See Remark \ref{uniqueness2} below. For larger sets $A$, we do not know whether the Gibbs measure $\mu$ satisfying \eqref{main1} is necessarily unique under the assumption $\|Q-\mathbbm{1}_{\{0\}}\|_{\frac{d+1}{2}} \leq \eta(d,|A|)$. 
By strengthening this assumption, we could get the following uniqueness statement: There exist positive numbers $\eta'(d,n)$ and $\delta(d,n)$ such that if $A$ is a finite subset of $S$ and $\|Q- \mathbbm{1}_{\{0\}} \|_{\frac{d+1}{2}} \leq \eta'(d,|A|)$ then the Markovian gradient specification which is induced by $Q$ on the regular $d$-tree with local state space $S$ has a {\em unique}  spatially homogeneous Markov-chain Gibbs measure $\mu$ whose single site marginal probability distribution $\pi_{\mu}$ satisfies
\[
\bigl\|\pi_{\mu} - {\textstyle \frac{1}{|A|}} \mathbbm{1}_A \bigr\|_1 < \delta(d,|A|).
\]
See Remark \ref{uniqueness1b} below.}
\end{rk}

\begin{rk}{\rm ({\em Affine independence)}  
Assuming that $\|Q-\mathbbm{1}_{\{0\}}\| \leq \eta(d,N)$ holds, the above theorem gives us a family of Gibbs measures $\{\mu_A\}_{A\in \mathcal{A}_N}$, where $\mathcal{A}_N$ denotes the set of all subsets $A$ of $S$ with $1\leq |A|\leq N$. Condition \eqref{main1} implies that these measures are pairwise distinct. More is actually true: none of the measures in the above family is a convex combination of the other ones, so each $\mu_A$ should be though as irreducible.  Indeed, this is a direct consequence of the fact that a non-trivial convex combination of spatially homogeneous Markov-chain Gibbs measures is never a Markov-chain Gibbs measure. In the case of a finite local state space $S$, this follows from Corollary 12.18 in \cite{Ge11}, but the proof directly generalizes to the case of a countably infinite state space, as all occurring sums are finite by the normalizability assumption on our boundary laws and all terms are strictly positive by the assumption of positivity of $Q$. In particular, we obtain that when $S$ is an infinite Abelian group and $\|Q-\mathbbm{1}_{\{0\}}\| \leq \eta(d,2)$ holds, then the convex set of all Gibbs measures $\mathcal{G}(\gamma)$ of the Gibbs specification induced by $Q$ is infinite dimensional also after modding out the action on it which is given by translations on $S$.}
\end{rk}

The proof of Theorem \ref{main1-thm} is based on an existence result for positive solutions $u\in \ell^{\frac{d+1}{d}}(S)$ of the normalized boundary law equation
\[
u = (Q * u)^d,
\]
which are suitably concentrated near the finite subset $A$. Boundary laws are discussed in Section \ref{bdrylaw-sec} and the proof of the existence result, which is based on a combined use of  the contraction mapping theorem and Brouwer's fixed point theorem, is discussed in Sections \ref{existence-sec} and \ref{lemmas-sec}. How to derive Theorem \ref{main1-thm} from this result is explained in Section \ref{proof-sec}.

\subsection{First applications}
\label{examples-sec}

The next two examples show how the assumption \eqref{mainass} translates for some concrete models. In the study of these models, we shall make use of the fact that the function $\eta$ satisfies the bounds
\begin{equation}
\label{bbb0}
\underline{c} \, d\, n^d \leq\eta(d,n)^{-(d+1)} \leq \overline{c}  \, d\, n^d \qquad \forall n\geq 1, \quad \forall d\geq 2,
\end{equation}
for suitable positive numbers $\underline{c}$ and $\overline{c}$, as proven in Lemma \ref{asymptotics} in Section \ref{existence-sec}.

\begin{ex}{\rm ({\em SOS model}) \label{SOS} Consider the case $S=\Z$ and $Q(i)=e^{-\beta |i|}$, where $\beta$ is a positive parameter modelling the inverse temperature.
Then
\[
\|Q - \mathbbm{1}_{\{0\}} \|_{\frac{d+1}{2}} = \left( \frac{2}{e^{\frac{d+1}{2} \beta}-1} \right)^{\frac{2}{d+1}},
\]
and the assumption \eqref{mainass} reads
\[
\beta \geq \underline{\beta}(d,n) := \frac{2}{d+1} \log \left( 1+ 2 \cdot \eta(d,n)^{-\frac{d+1}{2}} \right).
\]
Table \ref{table SOS} lists some approximate values of the threshold $\underline{\beta}$:

\medskip
\begin{table}[h]
	\centering
\begin{tabular}{l|l|l|l}
	$d$ & $n=1$& $n=2$ & $n=10$\\	\hline 
	$2$&$1.953$ &$2.367$  &$3.396$ \\
	$3$&$1.400$  &$1.870$ &$3.038$ \\
	$6$ &$0.810$  &$1.366$  & $2.719$
\end{tabular}
\caption{The threshold $\underline{\beta}(d,n)$ for the SOS model.}
\label{table SOS}
\end{table}
	
\medskip	

The asymptotic behaviour of this threshold for $d$ and/or $n$ tending to infinity can be determined as follows. By \eqref{bbb0}, $\underline{\beta}$ has the  bounds
\[
\frac{2}{d+1} \log \left( 1+ 2  \sqrt{\underline{c} d} n^{\frac{d}{2}} \right) \leq \underline{\beta}(d,n) \leq \frac{2}{d+1} \log \left( 1+ 2c \sqrt{\underline{c} d} n^{\frac{d}{2}} \right).
\]
By the inequalities
\[
\log x \leq \log (1+x)\leq \log x + \frac{1}{x}, \qquad \forall x> 0,
\]
we obtain the bounds
\[
\begin{split}
\frac{1}{d+1} \bigl( 2 \log (2 \sqrt{\underline{c}}) + \log d + d\log n \bigr) \leq \frac{1}{d+1} \left( 2 \log (2 \sqrt{\overline{c}}) +  \log d + d \log n + \frac{1}{\sqrt{\overline{c}}} \right),
\end{split}
\]
and hence
\[
\frac{\log d}{d} \violet{-} \frac{a}{d} + \frac{2}{3} \log n \leq \underline{\beta}(d,n)  \leq \frac{\log d}{d} + \frac{b}{d} + \log n \qquad \forall d\geq 2, \; \forall n\geq 1,
\]
for suitable positive numbers $a$ and $b$. In particular, for any fixed $n\in \N$ the threshold $\underline{\beta}(d,n)$ has size of the order $\log n$ for $d\rightarrow \infty$. When $n=1$, $\underline{\beta}(d,n)$ converges to zero and is asymptotic to $\frac{\log d}{d}$.
}
\end{ex}

\begin{ex}
{\rm ({\em Log potential}) Consider the case $S=\Z$ and $Q(i)=\frac{1}{(1+|i|)^{\beta}}$. Then 
\[
\|Q - \mathbbm{1}_{\{0\}} \|_{\frac{d+1}{2}}=\left(2 \zeta\bigl({\textstyle \frac{d+1}{2}} \beta\bigr)-2 \right)^{-\frac{2}{d+1}},
\]
where $\zeta$ denotes the Riemann zeta function. The assumption \eqref{mainass} now becomes
\begin{equation}
\label{log0}
\beta \geq \underline{\beta}(d,n) := \frac{2}{d+1} \zeta^{-1} \left( 1 + {\textstyle \frac{1}{2} \eta(d,n)^{\frac{d+1}{2}}} \right),
\end{equation}
where $\zeta^{-1}:(1,+\infty) \rightarrow (1,+\infty)$ denotes the inverse of the restriction of the  Riemann zeta function to the interval $(1,+\infty)$, on which this function is strictly monotonically decreasing  with image $(1,+\infty)$. Table \ref{table Log} lists some approximate values of the threshold $\underline{\beta}$:

\begin{table}
	\centering	
\begin{tabular}{l|l|l|l}
	$d$ & $n=1$& $n=2$ & $n=10$\\	\hline 
	$2$&$2.974$ &$3.527$  &$4.946$ \\
	$3$&$2.145$  &$2.773$ &$4.402$ \\
	$6$ &$1.240$  &$1.994$  & $3.924$
\end{tabular}
\caption{The threshold $\underline{\beta}(d,n)$ for the model with log potential.}
\label{table Log}
\end{table}
}

{\rm We now determine the asymptotics of $\underline{\beta}(d,n)$ for $d$ and/or $n$ tending to infinity. On the interval $(1,+\infty)$, the Riemann zeta functions satisfies the bounds
\begin{equation}
\label{log1}
1 + \frac{1}{2^s} \leq \zeta(s) = 1 + \frac{1}{2^s} + \sum_{k=3}^{\infty} \frac{1}{k^s} \leq 1 + \frac{1}{2^s} + \int_2^{\infty} \frac{dx}{x^s} = 1 + \left( 1 + {\textstyle \frac{2}{s-1} } \right) \frac{1}{2^s}.
\end{equation}
If $\overline{s}$ is the unique number in $(1,+\infty)$ such that $\zeta(\overline{s}) = \frac{3}{2}$, we have
\begin{equation}
\label{log2}
\zeta(s) \leq 1 + \frac{c}{2^s} \qquad \forall s\in [\overline{s},+\infty),
\end{equation}
where $c:= 1 + \frac{2}{\overline{s}-1}$. From the lower bound in \eqref{log1} we deduce the bound
\begin{equation}
\label{log3}
\zeta^{-1}(1+r) \geq \frac{\log \frac{1}{r}}{\log 2} \qquad \forall r\in (0,+\infty).
\end{equation}
Similarly, the upper bound \eqref{log2} implies
\begin{equation}
\label{log4}
\zeta^{-1}(1+r) \leq \frac{\log c + \log \frac{1}{r}}{\log 2}  \qquad \forall r\in \bigl(0,{\textstyle \frac{1}{2}} \bigr].
\end{equation}
By \eqref{log0}, \eqref{log3} and \eqref{bbb0} we find
\[
\begin{split}
\underline{\beta}(d,n) &\geq \frac{2}{(d+1) \log 2} \left( \log 2 + {\textstyle \frac{1}{2}} \log \eta^{-(d+1)} \right) \\ &\geq \frac{2}{(d+1) \log 2} \left( \log 2 + {\textstyle \frac{1}{2}} \log \underline{c} + {\textstyle \frac{1}{2}} \log d + {\textstyle \frac{d}{2}} \log n \right) \\ &\geq \frac{1}{\log 2} \left( \frac{\log d}{d} - \frac{a}{d} + \frac{2}{3} d\, \log n \right),
\end{split}
\]
for a suitable positive number $a$. Similarly, \eqref{log0}, \eqref{log4}, \eqref{bbb0} and the bound $\frac{1}{2} \eta(d,n)^{\frac{d+1}{2}}< \frac{1}{2}$ imply
\[
\begin{split}
\underline{\beta}(d,n) &\leq \frac{2}{(d+1) \log 2} \left(\log c + \log 2 + {\textstyle \frac{1}{2}} \log \eta^{-(d+1)} \right) \\ &\leq \frac{2}{(d+1) \log 2} \left(\log c + \log 2 + {\textstyle \frac{1}{2}} \log \overline{c} + {\textstyle \frac{1}{2}} \log d + {\textstyle \frac{d}{2}} \log n \right) \\ &\leq \frac{1}{\log 2} \left( \frac{\log d}{d} + \frac{b}{d} + d\, \log n \right),
\end{split}
\]
for a suitable number $b>0$. We conclude that the threshold $\underline{\beta}$ satisfies the lower and upper bounds
\[
\frac{1}{\log 2} \left( \frac{\log d}{d} - \frac{a}{d} + \frac{2}{3} d\, \log n \right) \leq \underline{\beta}(d,n) \leq \frac{1}{\log 2} \left( \frac{\log d}{d} + \frac{b}{d} + d\, \log n \right)
\]
for every $d\geq 2$ and $n\geq 1$.
Up to the multiplication by the factor $\frac{1}{\log 2}$, the asymtptotics of this threshold is analogous to the one we found in Example \ref{SOS} for the SOS  model. 
}
\end{ex}

\section{An application to the existence of delocalized gradient Gibbs measures}\label{sec GGM} 

In this section, we show how Theorem \ref{main1-thm} implies the existence of suitable \textit{gradient Gibbs measures} with \textit{height-dimension} $\Z$. The sets $A$ 
	which appear as a discrete parameter of the measures and 
	which played the role of localization sets for the \textit{Gibbs measures} of the previous section will now acquire a different role. Indeed, for the delocalized gradient Gibbs measures we discuss in this section, 
	there is no invariant single-site probability distribution in which the height variables would localize.
	Instead, the sets $A$ will govern the structure 
	of most probable increments along the edges, in a way that we will describe now.  We first review the necessary definitions.

\subsection{Gradient Gibbs measures}\label{subsec: DefGGM} 

The notion of a gradient Gibbs measure for lattice models has been established in \cite{FS97} and further exploited in \cite{Sh05}. In this subsection we present an adaption to the situation on the tree, which is based on \cite{KS17} and \cite{HeKu21a}. Consider the case $S = \Z$, in which we interpret a spin configuration $\omega \in \Omega$ as a height configuration and denote the local state space $\Z$ as the \textit{height-dimension} of the model.

Define the gradient projection $\nabla: \Omega \rightarrow \Z^{\vec{E}}; (\omega_x)_{x \in V} \mapsto (\omega_y-\omega_x)_{(x,y) \in \vec{E}}$.  Then
\[\Omega^\nabla:=\nabla(\Omega)=\{\zeta \in \Z^{\vec{E}} \mid \zeta_{(x,y)}=-\zeta_{(y,x)} \text{ for all } (x,y) \in \vec{E}\}\] is the set of all \textit{gradient configurations}.
For any oriented edge $b=(x,y) \in \vec{E}$ let $\eta_b: \Omega^\nabla \rightarrow S,\, \eta_b(\zeta):=\zeta_{b}$ denote the \textit{gradient spin projection} along $b$.  By construction, $\eta_{(x,y)} \equiv-\eta_{(y,x)}$ whenever $(x,y) \in \vec{E}$.
We endow $\Omega^\nabla$ with the product $\sigma$-algebra $\mathcal{F}^\nabla$ generated by all gradient spin projections, i.e.,  $\mathcal{F}^\nabla=\sigma(\eta_b \mid b \in \vec{E})$. 

Let $x_0 \in V$ be any fixed vertex. By connectedness of the tree and absence of cycles, prescription of any fixed height $s \in \mathbb{Z}$ at $x_0$ gives rise to a well-defined injective map 
\[
\Omega^\nabla\rightarrow \Omega, \qquad \zeta \mapsto \omega \quad \mbox{where } \nabla \omega = \zeta, \; \omega(x_0)=s.
\]
A gradient configuration on the tree can be thus considered as a \textit{relative height configuration} where two height configurations are equivalent iff one is obtained from the other one by a joint height shift $\theta_i(j):= j+i$. Hence we have the identification 
\[
\Omega^\nabla =\Omega /\Z. 
\]
Similar to the situation on the lattice \cite{Sh05}, we may think of $\mathcal{F}^\nabla$ (or more precisely the $\sigma$-algebra on $\Omega$ generated by $\nabla$) as the set of all events in $\mathcal{F}$ which are invariant under all joint height shifts $\theta_i$.

To lift the Gibbsian specification $\gamma$ for height configurations to a gradient specification $\gamma'$ for gradient configurations one has to consider that due to the absence of cycles on the tree the complement of any finite subtree $(\Lambda,E_\Lambda)$ decomposes into disjoint subtrees. This means that 
\[
\zeta_{\Lambda^c} \in \Omega^\nabla_{\Lambda^c}:=\{ \zeta \in \Z^{\vec{E}_\Lambda^c} \mid \zeta_{(x,y)}=-\zeta_{(y,x)} \text{ for all } (x,y) \in \vec{E}_\Lambda^c \}
\]
does not determine the relative heights at the boundary as an element of $\Z^{\partial \Lambda} /\Z$, i.e., up to a joint height shift at the boundary. 

Thus, the appropriate outer gradient $\sigma$-algebra $\mathcal{T}^\nabla_\Lambda$ has to implement both the information on the gradient spin variables outside $\Lambda$ and the information on the relative heights at the boundary. As the relative heights of the boundary are uniquely determined by the gradients inside $\Lambda \cup \partial \Lambda$ (each two vertices at the boundary are connected by a unique path in $\Lambda \cup \partial \Lambda$), these relative heights at the boundary can be expressed in terms of an $\mathcal{F}^\nabla$ measurable function $[\cdot]_{\partial \Lambda}: \Omega^\nabla \rightarrow \Z^{\partial \Lambda} / \Z$.
Hence
\begin{defn}[Definition 2.3 in \cite{KS17}]
The \textit{gradient-$\sigma$-algebra outside $\Lambda$} is defined as	
\begin{equation}
\mathcal{T}^\nabla_\Lambda:= \sigma((\eta_b)_{b \cap \Lambda = \emptyset}, \, [\eta]_{\partial \Lambda}  ) \subset \mathcal{F}^\nabla.
\end{equation}
\end{defn}
This allows to lift a specification $\gamma$ to a gradient specification $\gamma'$:
\begin{defn}[Definition 2.4 in \cite{KS17}]\label{def: GradientGibbs}
The \textit{gradient Gibbs specification}
	is defined as the family of probability kernels $(\gamma'_\Lambda)_{\Lambda \Subset V}$ from $(\Omega^\nabla, \mathcal{T}_\Lambda^\nabla)$ to $(\Omega^\nabla, \mathcal{F}^\nabla)$ given by
	\begin{equation}\label{grad}
	\int F(\rho) \gamma'_\Lambda(\text{d}\rho \mid \zeta) = \int F(\nabla \varphi) \gamma_\Lambda(\text{d}\varphi\mid\omega)
	\end{equation}
	for all bounded $\mathcal{F}^{\nabla}$-measurable functions $F$, where $\omega \in \Omega$ is any height configuration with $\nabla \omega = \zeta$.
\end{defn}
Finally, the DLR-equation for gradient measures on the tree reads:
\begin{defn}[Definition 2.5 in \cite{KS17}]
A measure $\eta \in \mathcal{M}_1(\Omega^\nabla)$ is called a \textit{gradient Gibbs measure (GGM)} if it satisfies the DLR equation 
\begin{equation}\label{eq: DLR}
\int \eta (d\zeta)F(\zeta)=\int \eta (\text{d}\zeta) \int \gamma'_{\Lambda}(\text{d}\tilde\zeta \mid \zeta) F(\tilde\zeta)
\end{equation}
for every finite subtree $(\Lambda,L_\Lambda)$ and for all bounded continuous functions $F$ on $\Omega^\nabla$.  
\end{defn}

\subsection{From Gibbs-measures for clock-models to integer-valued gradient Gibbs measures} 
\label{subsec: ConstructionOfGGM}

Consider the case $S=\Z$ and let $q \geq 2$ be an integer.
Assume that $Q \in \ell^1(\Z)$. Then 
\[
Q_q(\bar{i}):= \sum_{j \in \Z}Q(i+qj)
\] 
is a well defined function on the Abelian group $\Z_q:= \Z/q\Z$, which we think of as a "fuzzy" transfer operator on $\Z_q$. 
Then the Gibbsian specification $\gamma^q$ on $(\Z_q)^V$ associated with $Q_q$ via \eqref{eq: GibbsSpecification} describes a clock model.
As shown in \cite{KS17}, \cite{HeKu21a} and \cite{HeKu21}, any Gibbs measure on $(\Z_q)^V$ for $Q_q$ can be assigned an (integer-valued) gradient Gibbs measure on $\Omega^\nabla$.
In this subjection we briefly summarize the construction as described in \cite{HeKu21}.

For any $\bar{i} \in \Z_q$ define a conditional distribution $\rho_Q^q(\cdot \mid \bar{i})$ on $\Z$ equipped with the power set $\mathcal{P}(\Z)$ by
\begin{equation}\label{eq: rhokernel}
\rho_Q^q(j \mid \bar{i}):=\mathbbm{1}_{\{\bar{i}\}}(j) \frac{Q(j)}{Q^q(\bar{i})}.
\end{equation}
Then we can define a map $T_Q^q: \mathcal{M}_1(\Z_q^V, \mathcal{P}(\Z_q)^{\otimes{V}}) \rightarrow \mathcal{M}_1(\Omega^\nabla, \mathcal{F}^\nabla)$ from $q$-spin measures on vertices to integer-valued gradient measures in terms of the following two-step procedure:
\begin{equation} \label{eq: ConstructionOfGGM}
T^q_Q(\mu)(\eta_\Lambda=\zeta_\Lambda)=\sum_{\bar{\omega}_\Lambda \in \Z_q^\Lambda}\mu(\bar{\sigma}_\Lambda=\bar{\omega}_\Lambda)\prod_{\genfrac{}{}{0pt}{}{(x,y) \in {}^w\vec{L}}{x,y \in \Lambda}} \rho^q_{Q}(\zeta_{(x,y)} \mid \bar{\omega}_y-\bar{\omega}_x),
\end{equation}
where $\Lambda \subset V$ is any finite connected set and $w \in V$ is an arbitrary fixed vertex.

The assignment \eqref{eq: ConstructionOfGGM} describes a two-step procedure, where in the first step $\Z_q$-valued configurations are drawn from $\mu$ and in the second step  integer-valued gradients are edge-wise independently sampled conditioned on the $\Z_q$-valued increment along the respective edge. See also Figure \ref{Fig: ConstrGGM} below.

Then the following holds true without any assumption on spatial homogeneity:
\begin{thm}[Theorem 4.1 in \cite{KS17}, Theorem 2 in \cite{HeKu21} ]\label{thm: Gibbs property}
$T^q_Q$ maps Gibbs measures on $\Z_q^V$ for the fuzzy specification $\gamma^q$ to gradient Gibbs measures on $\Omega^\nabla$ for the gradient Gibbs specification $\gamma'$ \eqref{grad}.
\end{thm}
Note that the fact that Gibbs measures are mapped to gradient Gibbs measures as described in Theorem \cite{HeKrKu19} is a rare example for the preservation of the quasilocal Gibbs property, as it occurs throughout the whole phase diagram. In general, local maps tend to destroy the Gibbs property in strong coupling regions, see e.g., \cite{HeKrKu19,EFHR02}.
\begin{figure}[h]
\begin{minipage}{0.5\textwidth}
\begin{tikzpicture}[level 1/.style = {black, level distance = 1.5cm, sibling angle = 120},level 2/.style = {black,level distance =1.5cm, sibling angle = 70},
level 3/.style = {black,level distance =1cm, sibling angle = 50},
grow cyclic]
\node[circle,draw=black,thick,
inner sep=0pt,minimum size=4mm] {$w$}
child {node [circle, draw=black, thick, minimum size = 4mm, inner sep = 0] {$x$}
	child {node [circle, draw=white, thick, minimum size = 4mm, inner sep = 0] {}} 
	child {node [circle, draw=white, thick, minimum size = 4mm, inner sep = 0] {}}
	edge from parent [->]}
child {node [circle,draw=black,thick,
	inner sep=0pt,minimum size=4mm]{$y$}
	child {node [circle,draw=white,thick,
		inner sep=0pt,minimum size=4mm] {}}
	child {node [circle,draw=white,thick,
		inner sep=0pt,minimum size=4mm] {}}
	edge from parent [->] {}}
child {node [circle, draw=black, thick, minimum size = 4mm, inner sep = 0] {$z$}  
	child {node [circle, draw=white, thick, minimum size = 4mm, inner sep = 0] {}}
	child {node [circle, draw=white, thick, minimum size = 4mm, inner sep = 0] {}}
	edge from parent [->] {}};
\node (b1) at (-0.2,1.2) {\small$\bar{\omega}_z$};
\node (b2) at (0.3,0.4) {\small$\bar{\omega}_w$};
\node (b3) at (1.1,0.39) {\small$\bar{\omega}_y$};
\node (b4) at (-0.9,-0.9) {\small$\bar{\omega}_x$};
\node (c1) at (-1,0.5) {\small$\bar{\omega}_z-\bar{\omega}_w$};
\node (c2) at (0.7,-0.45) {\small$\bar{\omega}_y-\bar{\omega}_w$};
\node (d1) at (-0.8,0.15) {\small$\eta_{(w,z)}$};
\node (d2) at (0.6,-1) {\small$\eta_{(w,y)}$};
\node (e1) at (1,-1.5) {\small $\sim \rho_Q^q(\cdot \mid \bar{\omega}_y-\bar{\omega}_w)$};
\end{tikzpicture}
\end{minipage}
\begin{minipage}{0.45\textwidth}
\caption{\label{Fig: ConstrGGM}\small Construction of the measure $T_Q^q(\mu^q)$: In the first step, a $\Z_q$-valued configuration $\bar{\omega}$ is drawn from $\mu^q$. 
Conditional on the $\Z_q$-valued increment along the respective edge, the integer-valued gradient $\eta$ is then distributed with respect to $\rho_Q^q$ \eqref{eq: rhokernel}. 
}
\end{minipage}
\end{figure}
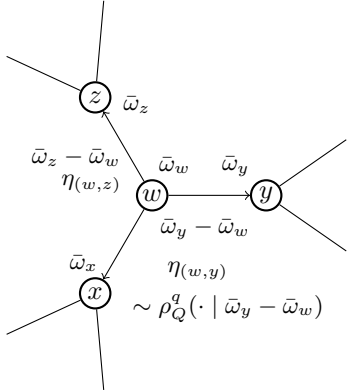

The map $T_Q^q$ as defined in \ref{eq: ConstructionOfGGM} has two important properties:
First, as we will see below, any integer-valued gradient Gibbs measure $\eta \in T_Q^q(\mathcal{G}(\gamma^q))$ is delocalized.

Second, for any gradient Gibbs measure $\nu^q$ which is given as the image of a homogeneous Markov-chain Gibbs measure $\mu^q$ on $\Z_q$ we can identify both the period $q$ and the distribution of the underlying Markov chain $\mu$ from  $\nu^q$ up to certain symmetries.
This motivates to call such a gradient Gibbs measure $\nu^q$ a \textit{delocalized gradient Gibbs measure of height-period $q$}.

The general delocalization statement of Theorem \ref{thm: deloc} below rests on Proposition 1 in \cite{HeKu21} in combination with extremal decomposition in $\mathcal{G}(\gamma^q)$. A proof is given at the end of this section. The less general identifiability result has already been proved in \cite{HeKu21a}. 

\begin{thm}\label{thm: deloc}
Let $q =2,3,\ldots$.
Then any $\nu$ in $T_Q^q(\mathcal{G}(\gamma^q)) \subset \mathcal{G}(\gamma')$ delocalizes in the sense that $\nu(W_n=k) \stackrel{n \rightarrow \infty}{\rightarrow} 0$ for any total increment $W_n$ along a path of length $n$ and any $k \in \Z$. 	
\end{thm}
Note that Theorem \ref{thm: deloc} holds without the assumption of homogeneity, while for the identifiability result below we have to restrict to homogeneous measures.
\begin{thm}[Theorem 5 and Corollary 1 in \cite{HeKu21a}] \label{thm: ident}
Let $q\geq 2$ be an integer.
Let $\nu^q \in T_Q^q(\mathcal{G}(\gamma^q))$ be such that $\nu^q=T_Q^q(\mu^q)$ for some homogeneous Markov-chain Gibbs measure $\mu^q$ on $\Z_q^V$. Then the period $q$ is uniquely determined by $\nu$ up to integer-valued multiples. Moreover, the distribution of $\mu^q$ is uniquely determined by $\nu^q$ up to a joint height shift $\theta_i$ on $\Z_q$. 
\end{thm}

\begin{proof}[Proof of Theorem \ref{thm: deloc}]
	By Proposition 1  in \cite{HeKu21}, we already know that for any (not necessarily homogeneous) $q$-state Markov-chain Gibbs measure $\mu^q \in \mathcal{G}(\gamma^q)$ the associated integer-valued gradient Gibbs measure $T_Q^q(\mu^q)$ delocalizes in the sense of $T_Q^q(\mu^q)(W_n=k) \stackrel{n \rightarrow \infty}{\rightarrow} 0$ for any fixed $k \in \Z$. 

	Now let $\mu \in \mathcal{G}(\gamma^q)$ be any Gibbs measure on $(\Z_q)^V$. By extremal decomposition, we have a unique probability measure $w_\mu$ on $(\text{ex}\,\mathcal{G}(\gamma^q), \text{ev}\,\text{ex} \, \mathcal{G}(\gamma^q))$, such that
	\begin{equation}
	\mu(\cdot)=\int_{\text{ex}\, \mathcal{G}(\gamma)}w_\mu(\text{d}\tilde{\mu})\tilde{\mu}(\cdot).
	\end{equation}
	Here,  $\text{ev}\,\text{ex} \, \mathcal{G}(\gamma)$ denotes the \textit{evaluation $\sigma$-algebra} on $\text{ex} \, \mathcal{G}(\gamma)$ generated by the \textit{evaluations} of the form $\pi_A: \tilde{\mu} \mapsto \tilde{\mu}(A)$,  where $A \in \mathcal{P}(\Z_q)^{V}$ is a fixed event.

	Let $n \in \mathbb{N}$ and $x,y \in V$ such that $\text{d}(x,y)=n$. Let $\Gamma(x,y)$ denote the shortest path connecting $x$ and $y$ and let $W_n$ denote the integer-valued total increment along $\Gamma(x,y)$ distributed with $T_Q^q(\mu)$.
	%As for any $k \in \Z$ the set \[\{\bar{\omega}_{\Gamma(x,y)} \in (\Z_q)^{\Gamma(x,y)} \mid \sum_{(u,v) \in \Gamma(x,y)} \bar{\omega}_v-\bar{\omega}_u \equiv k \mod q\Z\}\] is finite, 
	Recalling the definition of $T_Q^q$ in \eqref{eq: ConstructionOfGGM}, linearity of the integral gives
	\begin{equation}\label{eq: TqAffine}
	\begin{split}
	T_Q^q(\mu)(W_n=k)&=T_Q^q\left(\int_{\text{ex} \,\mathcal{G}(\gamma)}w_\mu(\text{d}\tilde{\mu})\tilde{\mu}\right)(W_n=k) \cr 
	&=\int_{\text{ex} \, \mathcal{G}(\gamma)}w_\mu(\text{d}\tilde{\mu})\left(T_Q^q(\tilde{\mu})(W_n=k) \right).
	\end{split}
	\end{equation}
	Now, any $\tilde{\mu} \in \text{ex}\, \mathcal{G}(\gamma^q)$ is a Markov chain. While this Markov chain is in general inhomogeneous, we can still employ  Proposition 1 in \cite{HeKu21}, which says that for any such $\tilde{\mu}$ we have
	$T_Q^q(\tilde{\mu})(W_n=k) \stackrel{n \rightarrow \infty}{\rightarrow} 0$. Hence, dominated convergence (eg. Corollary 6.26 in \cite{Kl14}) with integrable majorant $g(\tilde{\mu}) =1$ for all $\tilde{\mu} \in \text{ex}\, \mathcal{G}(\gamma)$ applied to  \eqref{eq: TqAffine} shows that $T_Q^q(\mu)(W_n=k) \stackrel{n \rightarrow \infty}{\rightarrow} 0$, which concludes the proof of Theorem \ref{thm: deloc}.
\end{proof}

\subsection{Existence of height-periodic gradient Gibbs measures} 

The existence result for localized Gibbs measures of Theorem \ref{main1-thm} above implies an existence criterion for an associated family of height-periodic gradient Gibbs measures: 

\begin{cor}\label{thm: PeriodicGGM}
Consider the $d$-regular tree with $d \geq 2$. Let the integer $q \geq 2$ be a fixed height-period and let $Q \in \ell^1(\Z)$ be a spatially homogeneous positive transfer operator normalized by $Q(0)=1$. Let $N \in \{1, \ldots, q-1\}$ and assume that the normalized fuzzy transfer operator $Q_q$ on $\Z_q$ satisfies
\[
\|Q_q - \mathbbm{1}_{\{0\}}\|_{\frac{d+1}{2}}  \leq \eta(d,N).
\]
Then for every $A\subset \Z_q$ with $1\leq |A|\leq N$ there exists a spatially homogeneous $q$-periodic delocalized gradient Gibbs measure $\nu$ of the form 
\[
\nu=T_{Q_q}(\mu), 
\]
where $\mu$ is the homogeneous Markov-chain Gibbs measure on $\Z_q$ with lazy set $A$ given by Theorem \ref{main1-thm}.
\end{cor}

An illustration of the construction of such a gradient Gibbs measure $\nu$ in the case $q=5$ is given in Figure \ref{fig: nuqQ} below.

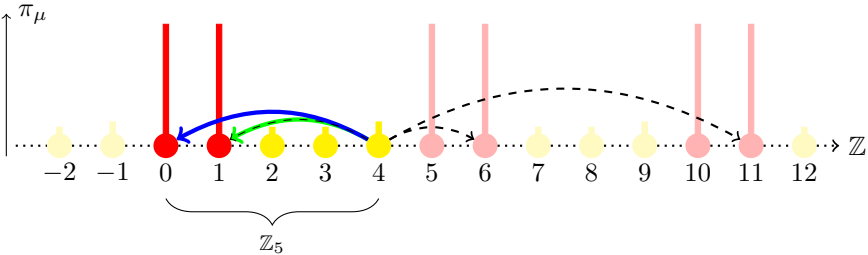
\begin{figure}[h]
	\centering
		\begin{tikzpicture}[scale=.7]
\node (ya1) at (-4.5,3.8) {};
\node (ya2) at (-4.5,6.5) {};
\node (yl) at (-7.5,6.5) {$\pi_{\mu}$};		
\node (g0) at (-8,4) {};
\node (g1) at (8,4) {$\Z$};		
\node (c1) at (-7,4) [circle, draw=yellow!30,thick,fill=yellow!30,minimum size = 3mm, inner sep = 0]{};	
\node at (-7,3.5) {$-2$};	
\node (c2) at (-6,4) [circle, draw=yellow!30,thick,fill=yellow!30,minimum size = 3mm, inner sep = 0]{};
\node at (-6,3.5) {$-1$};	
\node (c3) at (-5,4) [circle, draw=red,thick,fill=red,minimum size = 3mm, inner sep = 0]{};
\node at (-5,3.5) {$0$};	
\node (c4) at (-4,4) [circle, draw=red,thick,fill=red,minimum size = 3mm, inner sep = 0]{};
\node at (-4,3.5) {$1$};	
\node (c5) at (-3,4) [circle, draw=yellow,thick,fill=yellow,minimum size = 3mm, inner sep = 0]{};
\node at (-3,3.5) {$2$};	
\node (c6) at (-2,4) [circle, draw=yellow,thick,fill=yellow,minimum size = 3mm, inner sep = 0]{};
\node at (-2,3.5) {$3$};	
\node (c7) at (-1,4) [circle, draw=yellow,thick,fill=yellow,minimum size = 3mm, inner sep = 0]{};
\node at (-1,3.5) {$4$};	
\node (c8) at (0,4) [circle, draw=red!30,thick,fill=red!30,minimum size = 3mm, inner sep = 0]{};
\node at (0,3.5) {$5$};	
\node (c9) at (1,4) [circle, draw=red!30,thick,fill=red!30,minimum size = 3mm, inner sep = 0]{};
\node at (1,3.5) {$6$};	
\node (c10) at (2,4) [circle, draw=yellow!30,thick,fill=yellow!30,minimum size = 3mm, inner sep = 0]{};
\node at (2,3.5) {$7$};	
\node (c11) at (3,4) [circle, draw=yellow!30,thick,fill=yellow!30,minimum size = 3mm, inner sep = 0]{};
\node at (3,3.5) {$8$};	
\node (c12) at (4,4) [circle, draw=yellow!30,thick,fill=yellow!30,minimum size = 3mm, inner sep = 0]{};
\node at (4,3.5) {$9$};	
\node (c13) at (5,4) [circle, draw=red!30,thick,fill=red!30,minimum size = 3mm, inner sep = 0]{};
\node at (5,3.5) {$10$};	
\node (c14) at (6,4) [circle, draw=red!30,thick,fill=red!30,minimum size = 3mm, inner sep = 0]{};
\node at (6,3.5) {$11$};	
\node (c15) at (7,4) [circle, draw=yellow!30,thick,fill=yellow!30,minimum size = 3mm, inner sep = 0]{};
\node at (7,3.5) {$12$};	
\draw[draw=yellow!30,fill=yellow!30] (-7.05,4.2) rectangle ++(0.1,0.15);
\draw[draw=yellow!30,fill=yellow!30] (-6.05,4.2)rectangle ++(0.1,0.25);
\draw[draw=red,fill=red] (-5.05,4.2)  rectangle ++(0.1,2.1);
\draw[draw=red,fill=red] (-4.05,4.2) rectangle ++(0.1,2.1);
\draw[draw=yellow,fill=yellow] (-3.05,4.2) rectangle ++(0.1,0.15);
\draw[draw=yellow,fill=yellow] (-2.05,4.2) rectangle ++(0.1,0.15);
\draw[draw=yellow,fill=yellow] (-1.05,4.2) rectangle ++(0.1,0.25);
\draw[draw=red!30,fill=red!30] (-0.05,4.2) rectangle ++(0.1,2.1);
\draw[draw=red!30,fill=red!30] (0.95,4.2) rectangle ++(0.1,2.1);
\draw[draw=yellow!30,fill=yellow!30] (1.95,4.2) rectangle ++(0.1,0.15);
\draw[draw=yellow!30,fill=yellow!30] (2.95,4.2) rectangle ++(0.1,0.15);
\draw[draw=yellow!30,fill=yellow!30] (3.95,4.2)rectangle ++(0.1,0.25);
\draw[draw=red!30,fill=red!30] (4.95,4.2) rectangle ++(0.1,2.1);
\draw[draw=red!30,fill=red!30] (5.95,4.2) rectangle ++(0.1,2.1);
\draw[draw=yellow!30,fill=yellow!30] (6.95,4.2) rectangle ++(0.1,0.15);
\foreach \from/\to in 
{g0/c1,c1/c2,c2/c3,c3/c4,c4/c5,c5/c6,c6/c7,c7/c8,c8/c9,c9/c10,c10/c11,c11/c12,c12/c13,c13/c14,c14/c15}
\draw[dotted,thick] (\from) to (\to);
\draw[->,dotted,thick] (c15)--(g1);
\draw[->] (-8,3.8)--(-8,6.5);
\path[->,ultra thick,green] (c7) edge [bend right] node {} (c4);
\path[->,ultra thick,blue] (c7) edge [bend right] node {} (c3);
\path[->,thin,black,dashed] (c7) edge [bend right] node {} (c4);
\path[->,thick,black,dashed] (c7) edge [bend left] node {} (c9);
\path[->,thick,black,dashed] (c7) edge [bend left] node {} (c14);
\draw [decorate,decoration={brace,amplitude=10pt},xshift=0cm,yshift=-0.5cm]
(-1,3.5)--(-5,3.5)  node [black,midway,yshift=-0.6cm]
{\footnotesize $\Z_5$};
\end{tikzpicture}
\caption{\label{fig: nuqQ} \small The gradient Gibbs measure $\nu$ associated to the subset $A:=\{0,1\} \subset \Z_5$: two main transitions of the fuzzy chain $\mu$ in blue and green and the distribution of single-site marginals $\pi_{\mu}$ of $\mu$ concentrated on $\{0,1\}$. The green jump from $4$ to $1$ of the chain $\mu$ allows jumps of height $-3+5\Z$ for $\nu$, whose conditional distribution is according to $\rho^5_Q(\cdot \mid \bar{2})$ (see \eqref{eq: rhokernel}). Three of these possible jumps are marked by the dashed black arrows.}
\end{figure}	

How do the period $q$ and the concrete choice of the lazy set $A \subset \Z_q$ affect the associated gradient Gibbs measure $\nu_A^q$? The answer lies in (the proof of) Theorem \ref{thm: ident} above:
Considering the sequence of empirical distributions of increments along a branch of the tree gives in particular the sequence of empirical distributions of increments of the homogeneous fuzzy chain $\mu_A^q$. By the ergodic Theorem for Markov chains this sequence converges. The knowledge of the limit is equivalent to the knowledge of the stationary 
distribution on $\Z_q$ modulo cyclic shift, from which the set A can be read off. In particular, also the period $q$ can be recovered. For more details, see also the proof of Corollary 1 in \cite{HeKu21}. 

More can be said in the present case. 
	Consider the joint empirical distribution along a branch of the tree $x_1,x_2,\dots$ 
	for fuzzy spin values and 
	integer-valued increments
	of the form 
	\begin{equation}
	\label{eq: empirical-d}
	\frac{1}{n}\sum_{i=1}^n \delta_{\bar{\sigma}_{x_i},\sigma_{x_{i+1}}-\sigma_{x_{i}}}
	\end{equation}
	which is a random measure on  $\Z_q\times \Z$. It is important in the case 
	of delocalized gradient Gibbs measures to consider fuzzy spins $\bar{\sigma}_{x_i}$ in the first entry, 
	as the empirical measures for spins $\sigma_{x_i}$ would not converge. 
	
We claim that there is the $\nu_A^q$-a.s. convergence
	\begin{equation}
	\label{eq: empirical-dfuz}
	\Bigl(\frac{1}{n}\sum_{i=1}^n \delta_{\bar{\sigma}_{x_i},\sigma_{x_{i+1}}-\sigma_{x_{i}}}\Bigr)(\bar{a},c)
	\stackrel{n \rightarrow \infty}{\longrightarrow}\frac{\pi_{\mu^q_A}(\bar{a}) \pi_{\mu^q_A}(\bar{a}+\bar{c})^{\frac{d}{d+1}}}{
		Q^q * \pi_{\mu^q_A}^{\frac{d}{d+1}}(\bar{a})
	}Q(c)
	\end{equation}
Before we prove this statement, let us discuss what it tells us about the correlation structure 
	of the gradient state. 
	First note that  jump probabilities of increment size $c$, for fixed mod $q$ fuzzy classes $\bar{a},\bar{c}$ 
	depend only on the multiplicative factor $Q(c)$, 
	which strongly suppresses large jumps. On the other hand,  
	recall that by the concentration bounds of Theorem 3.1 the mod $q$ fuzzy measure
	$\pi_{\mu^q_A}$ concentrates strongly 
	on the set $A\subset \Z_q$ where it equals up to small errors the equidistribution. 
	So, \eqref{eq: empirical-dfuz} means that 
	the delocalized measure $\nu_A^q$ inherits a structure from the underlying measure 
	$\mu_A^q$, in which fuzzy jumps occur 
	mostly from $A$ to $A$, while arbitrarily large jumps in $\Z$ occur, but are suppressed and 
	modulated via the summable $Q$.  
	An example is discussed in Figure \ref{fig: nuqQ} above.

Finally, to prove the a.s. convergence \eqref{eq: empirical-dfuz} 
for a fixed pair $(\bar a, c)$, denote by $\bar c$ for the mod-$q$ class 
of $c$ and use the hidden Markov model  structure \ref{eq: ConstructionOfGGM} of the gradient measure to write 
the l.h.s. of \eqref{eq: empirical-dfuz} in the product form 
	\begin{equation}
	\label{eq: empirical-dfuz2}
	\frac{|\Lambda_n(\bar a, \bar c)|}{n} \times \frac{1}{|\Lambda_n(\bar a, \bar c)|}\sum_{i\in \Lambda_n(\bar a, \bar c)}
	Y_{i}(c)
	\end{equation}
where $\Lambda_n(\bar a, \bar c)=\{1\leq i \leq n, 
(\bar \sigma_i, \bar \sigma_{i+1}-\bar \sigma_i)=(\bar a,\bar c)\}$.  Here 
the variables $Y_i(c)$ are independent Bernoulli with success probability $\rho_Q^q(c \mid \bar{c})=Q(c)/Q^q(\bar c)$. 

By the Birkhoff a.s. ergodic theorem applied 
to the first factor in \eqref{eq: empirical-dfuz2} which we recognize as the 
pair empirical distribution of the irreducible hidden Markov chain  $\mu_A^q$, there is a set 
of full measure for $\mu_A^q$ (and hence for $\nu_A^q$) such that the first term in the product 
converges to its expectation. On this full measure set, in particular $|\Lambda_n(\bar a, \bar c)|\uparrow\infty$ by 
positivity of $Q^q$, and conditionally on that 
we can apply the SLLN for the independent variables $Y_j(c)$ to see that also the second term 
converges to its expectation $\frac{Q(c)}{Q^q(\bar{c})}$. Plugging in these expectations the claimed 
a.s. limit of \eqref{eq: empirical-dfuz}
follows.

\section{Proof of Theorem \ref{main1-thm}}
\label{sec: Proofs}

\subsection{Boundary laws and Gibbs measures}
\label{bdrylaw-sec}

As established in \cite{Z83}, tree-indexed Markov-chain Gibbs measures for nearest-neighbour interactions and a countable local state space can be described in terms of the solutions to a recursive system of boundary law equations on the tree. In this subsection, we briefly outline this formalism for the specific case of spatially homogeneous Gibbs measures for gradient interactions on the $d$-regular tree.

\begin{defn}
\label{defn:bdrylaw}
A {spatially homogeneous boundary law} for a transfer operator $Q$ is a positive function $u\in \ell^{\frac{d+1}{d}}(S)$ such that 
\begin{equation}
\label{eq: boundarylaw}
u =c \, (Q* u)^d
\end{equation}
for some $c>0$.
\end{defn}	
	
\begin{rk}
\label{rk: ImplicitNormalization}
{\rm If $(u,c)$ is a solution of \eqref{eq: boundarylaw} and $a$ is any positive number, then $v:=au$ satisfies
\[
v=c'(Q \ast v)^d
\] 
with $c'=a^{1-d}c$, and hence $v$ is also a boundary law. Boundary laws differing by a multiplicative constant are considered to be equivalent.  By multiplying $u$ by a suitable constant, we can always assume that $c=1$ in \eqref{eq: boundarylaw}.}
\end{rk}

Now, the relation between boundary laws and tree-indexed Markov chains reads:

\begin{thm}[See Theorem 3.2 in \cite{Z83}] \label{thm: Zachary}
Let $Q$ be a transfer operator. Then for the Markov specification $\gamma$ associated to $Q$  we have:
\begin{enumerate}[(i)]
\item Each spatially homogeneous boundary law $u$ for $Q$ defines a unique spatially homogeneous tree-indexed Markov-chain Gibbs measure $\mu \in \mathcal{MG}(\gamma)$ with marginals
\begin{equation}\label{BoundMC}
\mu(\sigma_{\Lambda \cup \partial \Lambda}=\omega) = \frac{1}{Z_\Lambda} \prod_{y \in \partial \Lambda} u (\omega_y) \prod_{\genfrac{}{}{0pt}{}{\{x,y\} \in L}{\{x,y\} \cap \Lambda \neq \emptyset}} Q(\omega_y-\omega_x),
\end{equation}
for any connected set $\Lambda \Subset V$ and $\omega\in S^{\Lambda \cup \partial \Lambda}$,
where $Z_{\Lambda}$ is the normalization constant which turns $\mu$  into a probability measure.
\item Conversely, every spatially homogeneous tree-indexed Markov-chain Gibbs measure $\mu \in \mathcal{MG}(\gamma)$ admits a representation of the form (\ref{BoundMC}) in terms of a spatially homogeneous boundary law $u$ which is unique up to a constant positive factor.
\end{enumerate}
\end{thm}

We note that the boundary law equation \label{eq: bl} guarantees that \eqref{BoundMC} describes a projective family of finite-volume marginals, whereas the summability condition $u\in \ell^{\frac{d+1}{d}}(S)$ gives us the finiteness of these finite-volume marginals.

From \eqref{BoundMC} and \eqref{eq: boundarylaw}, we can easily determine the single-site marginals and the transition matrices of the spatially homogeneous Gibbs measure that is determined by a boundary law:

\begin{pro} 
\label{useful}
Let $u$ be a spatially homogeneous boundary law for the transfer operator $Q$ and let $\mu$ be the corresponding spatially homogeneous tree-indexed Markov-chain Gibbs measure. Then
\[
\pi_{\mu}(i) = \frac{u(i)^{\frac{d+1}{d}}}{\|u^{\frac{d+1}{d}}\|_1}, \qquad P_{\mu} (i,j) = \frac{u(j) Q(i-j)}{(Q* u)(i)},
\]
for every $i,j\in S$.
\end{pro}

\begin{rk}
\label{transitionfor}
{\rm From the identities of Proposition \ref{useful} we deduce the formula
\[
P_{\mu}(i,j) = \frac{\pi_{\mu}(j)^{\frac{d}{d+1}} Q(i-j)}{\bigl( Q * \pi_{\mu}^{\frac{d}{d+1}} \bigr) (i)}
\]
which we used in \eqref{transitionfor0}.
}
\end{rk}

\subsection{Existence of solutions of the boundary law equation}
\label{existence-sec}

Let $d\geq 2$ be a positive integer and $Q\in\ell^{\frac{d+1}{2}}(S)$ be a positive function, which we normalize by assuming that $Q(0)=1$.
In this section, we wish to discuss the existence of positive solutions $u\in \ell^{\frac{d+1}{d}}(S)$ of the normalized boundary law equation
\[
u = (Q * u)^d.
\]
It  is convenient to set $u=x^d$ and rewrite the above equation as
\begin{equation}
\label{bdrylaw}
x=Q* x^d,
\end{equation}
where $x$ is a positive element of $\ell^{d+1}(S)$. We split $Q$ as
\[
Q = \mathbbm{1}_{\{0\}} + q,
\]
and rewrite \eqref{bdrylaw} as
\begin{equation}
\label{bdrylaw2}
x - x^d = q * x^d.
\end{equation}
The reformulation \eqref{bdrylaw2} shows that every positive solution $x$ takes values in the interval $(0,1)$. 

We  fix a finite subset $A\subset S$ and look for solutions $x\in \ell^{d+1}(S)$ of \eqref{bdrylaw2} which are close to $1$ on $A$ and close to 0 on its complement $A^c$. More precisely, we denote by 
\[
\lambda_d := d^{-\frac{1}{d-1}} \in (0,1)
\]
the point at which the function
\[
\varphi_d:[0,+\infty) \rightarrow \R, \qquad \varphi_d(r) = r - r^d,
\]
achieves its maximum and look for solutions $x\in \ell^{d+1}(S)$ of \eqref{bdrylaw2} such that
\begin{equation}
\label{where}
A = \{ \lambda_d < x< 1\}, \qquad A^c = \{ 0< x < \lambda_d \}.  
\end{equation}

We recall that in Section \ref{main1-sec} we defined $\rho=\rho(d,n)$  to be the unique positive number such that
\begin{equation}
\label{defrho}
(d-1)\rho^{d+1} + dn \rho^{d-1} - n = 0,
\end{equation}
and $\eta=\eta(d,n)\in (0,1)$ to be the number
\begin{equation}
\label{defnu}
\eta(d,n) =  \frac{\rho-\rho^d}{(\rho^{d+1} + n)^{\frac{d}{d+1}}}.
\end{equation}
Here is our existence result for solutions of \eqref{bdrylaw} satisfying the conditions \eqref{where}.

\begin{pro}
\label{existence}
Let $Q\in \ell^{\frac{d+1}{2}}(S)$ be a positive function with $Q(0)=1$ and set $q := Q - \mathbbm{1}_{\{0\}}$. Assume that
\begin{equation}
\label{ass}
\|q\|_{\frac{d+1}{2}} \leq \eta(d,n) 
\end{equation}
for some integers $d\geq 2$ and $n\geq 1$.
Then for every subset $A\subset S$ with $|A|=n$ there exists a positive function $\overline{x} \in \ell^{d+1}(S)$ such that
\begin{equation}
\label{aim}
 \overline{x}= Q* \overline{x}^d, \qquad
\|\overline{x}|_{A^c}\|_{\infty} \leq \|\overline{x}|_{A^c}\|_{d+1} < \rho(d,n) < \lambda_d < \overline{x}|_A < 1.
\end{equation}
Moreover:
\begin{enumerate}[(i)]
\item $\|\overline{x}|_{A^c}\|_{d+1} \leq \frac{\rho(d,n)}{\eta(d,n)} \|q\|_{\frac{d+1}{2}}$;
\item $0< 1 - \overline{x}|_A \leq \frac{d^{\frac{d}{d-1}}-d}{d-1} (1+n)^{\frac{d}{d+1}} \|q\|_{\frac{d+1}{2}}$.
\item $\overline{x}(i) \geq (1- \frac{d}{d-1} ( d^{\frac{d}{d-1}}-d) (1+n)^{\frac{d}{d+1}} \epsilon) \sum_{j\in A} Q(i-j)$ for every $i\in A^c$.
\end{enumerate}
\end{pro}

\begin{proof}
Given functions $x_0:A^c \rightarrow \R$ and $x_1: A \rightarrow \R$, we denote by
\[
x_0 \sqcup x_1 : S \rightarrow \R
\]
the function mapping $i\in A^c$ to $x_0(i)$ and $i\in A$ to $x_1(i)$. We start by fixing an arbitrary $x_1\in [\lambda_d,1]^A$ and look for functions $x$ of the form $x=x_0\sqcup x_1$ which solve \eqref{bdrylaw} on $A^c$, i.e.\
\begin{equation}
\label{bdrylawA^c}
x_0 = \bigl( Q* (x_0^d\sqcup x_1^d) \bigr)\bigr|_{A^c}.
\end{equation}
Equivalently, we are looking for the fixed points of the map
\[
F_{x_1} : \ell^{d+1}(A^c) \rightarrow \ell^{d+1}(A^c), \qquad F_{x_1}(x_0) = \bigl( Q * (x_0^d \sqcup x_1^d)\bigr)\bigr|_{A^c},
\]
which is well defined because of the Young inequality
\[
\|Q * (x_0^d \sqcup x_1^d)\|_{d+1} \leq \|Q\|_{\frac{d+1}{2}} \|x_0^d \sqcup x_1^d\|_{\frac{d+1}{d}} =  \|Q\|_{\frac{d+1}{2}} \|x_0 \sqcup x_1\|_{d+1}^d.
\]
Given $r>0$, set 
\[
X_r := \{ x_0\in \ell^{d+1}(A^c) \mid x_0\geq 0, \; \|x_0\|_{d+1} \leq r \}.
\]
We now check which condition on $r$ guarantees that $F_{x_1}$ maps $X_r$ to itself. If $x_0$ is in $X_r$ then $F_{x_1}(x_0)\geq 0$ and using again the Young inequality we find
\[
\begin{split}
\|F_{x_1}&(x_0) \|_{d+1} = \| x_0^d + q*(x_0^d \sqcup x_1^d)\|_{d+1} \leq \|x_0^d\|_{d+1} + \|q*(x_0^d \sqcup x_1^d)\|_{d+1}\\ &\leq \|x_0\|_{d(d+1)}^d + \|q\|_{\frac{d+1}{2}} \|x_0^d \sqcup x_1^d\|_{\frac{d+1}{d}} \leq \|x_0\|_{d+1}^d + \|q\|_{\frac{d+1}{2}} \|x_0\sqcup x_1\|_{d+1}^d \\ &=  \|x_0\|_{d+1}^d + \|q\|_{\frac{d+1}{2}} \bigl( \|x_0\|_{d+1}^{d+1} +  \|x_1\|_{d+1}^{d+1} \bigr)^{\frac{d}{d+1}} \\ &\leq r^d + \|q\|_{\frac{d+1}{2}}  (r^{d+1} + |A|)^{\frac{d}{d+1}} = r^d + \|q\|_{\frac{d+1}{2}}  (r^{d+1} + n)^{\frac{d}{d+1}} ,
\end{split}
\]
where we have also used the inequality $|x_1|\leq 1$ and the following consequence of the monotonicity of the $\ell^p$ norms: $\|x_0\|_{d(d+1)} \leq \|x_0\|_{d+1}$. Therefore, $F_{x_1}$ maps $X_r$ to itself provided that
\[
r^d + \|q\|_{\frac{d+1}{2}}  (r^{d+1} + n)^{\frac{d}{d+1}}  \leq r.
\]
This condition can be equivalently rewritten as
\begin{equation}
\label{selfmap}
\|q\|_{\frac{d+1}{2}} \leq f_{d,n}(r), 
\end{equation}
where $f_{d,n}:[0,+\infty) \rightarrow \R$ is the function  
\begin{equation}
\label{deff}
f_{d,n}(r) := \frac{r-r^d}{(r^{d+1} + n)^{\frac{d}{d+1}}}.
\end{equation}
Next note that
\begin{equation}
\label{dec}
F_{x_1}(x_0) = F_0(x_0) + (q*(0_{A^c} \sqcup x_1^d))|_{A^c},
\end{equation}
where $0_{A^c}$ denote the zero function on $A^c$. The map $F_0$ is the composition of the maps
\[
\ell^{d+1}(A^c) \rightarrow \ell^{\frac{d+1}{d}}(A^c), \qquad x_0 \mapsto x_0^d,
\]
and
\[
 \ell^{\frac{d+1}{d}} (A^c) \rightarrow \ell^{d+1}(A^c), \qquad y_0 \mapsto (Q*(y_0\sqcup 0_A)|_{A^c}.
 \]
By the mean value theorem, the first map has Lipschitz constant $dr^{d-1}$ on the $r$-ball of $\ell^{d+1}(A^c)$. The second map is linear with operator norm not exceeding
\[
\|Q\|_{\frac{d+1}{2}} = \left( 1 +  \|q\|_{\frac{d+1}{2}}^{\frac{d+1}{2}}\right)^{\frac{2}{d+1}}.
\]
Therefore, the restriction of the map $F_{x_1}$ to $X_r$ is a contraction if $r$ satisfies \eqref{selfmap} and
\[
d  \left( 1 +  \|q\|_{\frac{d+1}{2}}^{\frac{d+1}{2}}\right)^{\frac{2}{d+1}} r^{d-1}<1.
\]
This condition forces $r$ to belong to the interval $(0,\lambda_d)$ and can be equivalently rewritten as
\begin{equation}
\label{contraction}
\|q\|_{\frac{d+1}{2}} < g_d(r) ,
\end{equation}
where $g_d: (0,\lambda_d] \rightarrow \R$ is the function
\begin{equation}
\label{defg}
g_d(r) := \left( \Bigl( \frac{\lambda_d}{r} \Bigr)^{\frac{d^2-1}{2}} - 1\right)^{\frac{2}{d+1}}.
\end{equation}
The next lemma describes some useful properties of the functions $f_{d,n}$ and $g_d$. See also Figure \ref{Fig: fdnAndgd} for an illustration.
\begin{figure}[h]
\subfloat[$r \in (0,1)$]
{\includegraphics[width=6cm]{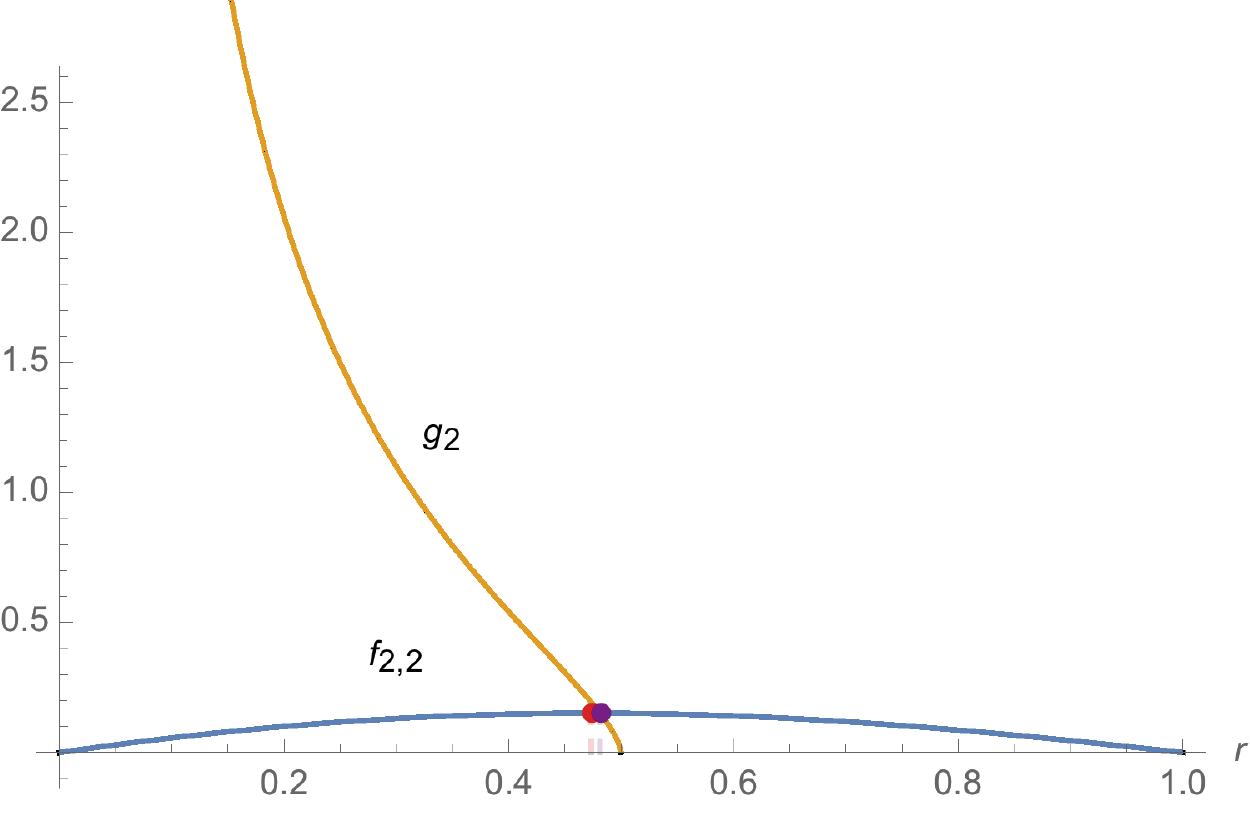}
} 
\subfloat[Zoomed in]	
{\includegraphics[width=6cm]{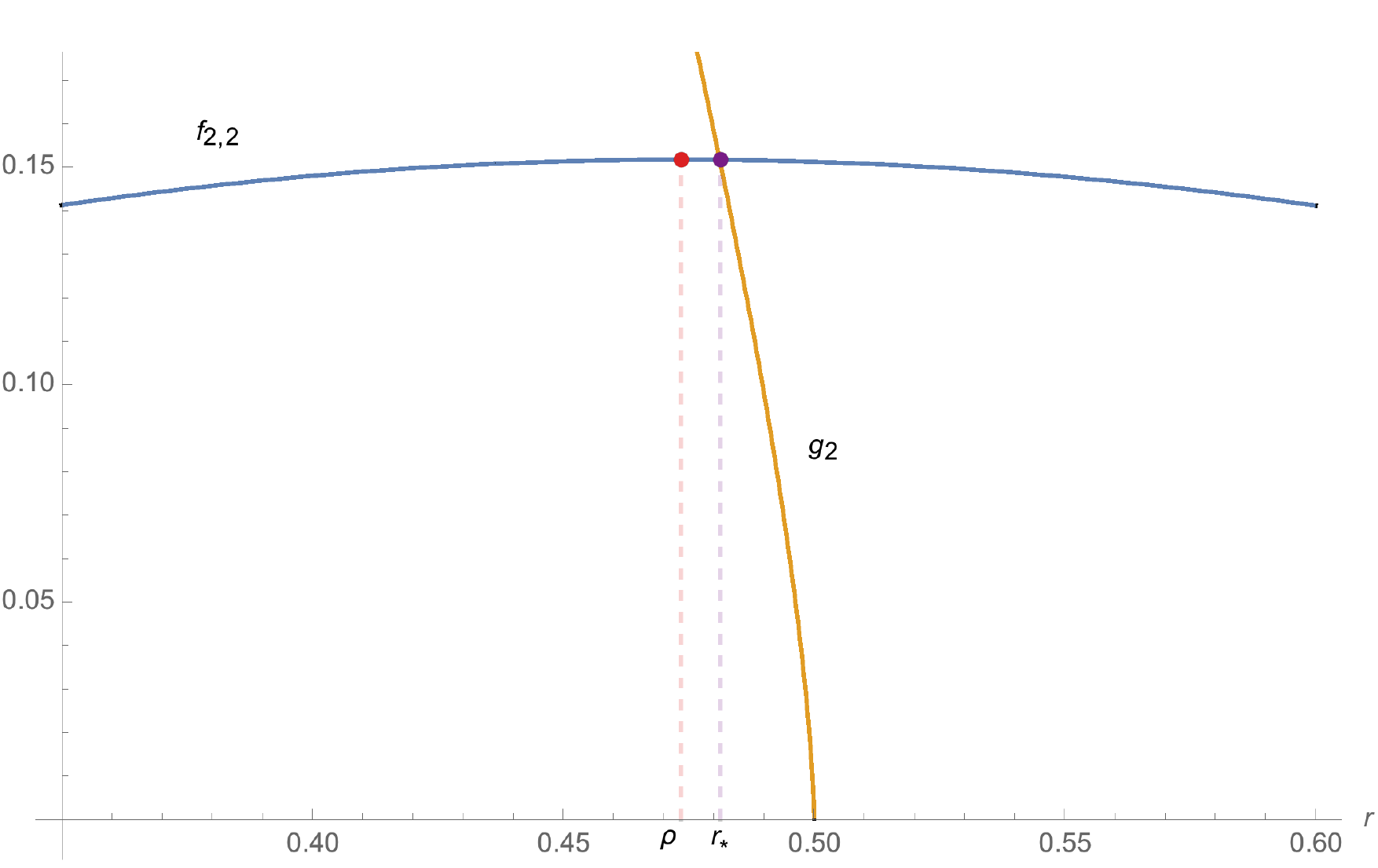}
}
\caption{Plot (a) shows the graphs of the functions $f_{d,n}$ and $g_d$ in the case $d=n=2$. Here, $\lambda_2=0.5$. The red point is the unique maximum of $f_{2,2}$, obtained at $\rho \approx 0.473$. The violet point marks the unique point of intersection of the two graphs, which happens at $r_* \approx 0.481$. Hence, $\eta(2,2)=f_{2,2}(\rho) \approx  0.152$. Plot (b) is a zoomed-in version of plot (a), where we see that the number $\rho$ where $f_{2,2}$ achieves its maximum is smaller than the number $r_*$ at which $f_{2,2}$ and $g_2$ coincide.}
\label{Fig: fdnAndgd}
\end{figure}

\begin{lemma}
\label{confronto}
Let $\rho=\rho(d,n)$ be the unique positive solution of the equation (\ref{defrho}). Then $\rho$ belongs to the interval $(0,\lambda_d)$. The function $f_{d,n}$ is strictly increasing on the interval $[0,\rho]$, strictly decreasing on $[\rho,+\infty)$ and strictly concave on $[0,1]$. Its maximum is the number
\[
\eta(d,n) = f_{d,n}(\rho(d,n))
\]
which is introduced in \eqref{defnu}. The function $g_d$ is strictly decreasing on $(0,\lambda_d]$  and there exists a number $r_* \in (\rho,\lambda_d)$ such that
\[
f_{d,n} < g_d \quad \mbox{on } (0,r_*), \qquad g_d < f_{d,n} \quad \mbox{on } (r_*,\lambda_d].
\]
\end{lemma}

Postponing the proof of this lemma to the next section, we proceed with the proof of Proposition \ref{existence}. By the above lemma and our assumption \eqref{ass}, we can find a number $r_q\in [0,\rho]$ such that
\[
f_{d,n}(r_q) = \|q\|_{\frac{d+1}{2}}.
\]
Then the equality holds in \eqref{selfmap} with $r=r_q$ and hence $F_{x_1}$ maps $X_{r_q}$ to itself. Since $r_q$ belongs to $[0,\rho]$, the above lemma implies that
\[
\|q\|_{\frac{d+1}{2}} = f_{d,n}(r_q) < g_d(r_q),
\]
so $r=r_q$ satisfies \eqref{contraction}. We conclude that $F_{x_1}$ is a contraction on $X_{r_q}$ and hence has a unique fixed point $\xi_0(x_1)$ in $X_{r_q}$. In particular, we have
\[
\|\xi_0(x_1)\|_{\infty} \leq \|\xi_0(x_1)\|_{d+1} \leq r_q < \rho(d,n),
\]
so $x_0=\xi_0(x_1)$ takes values in $[0,\rho(d,n))$ and is a solution of \eqref{bdrylawA^c}.

Note that by \eqref{dec}
\[
\xi_0(x_1) = (\mathrm{id} - F_0)^{-1}\bigl( (q*(0_{A^c} \sqcup x_1^d))|_{A^c} \bigr),
\]
where the map $\mathrm{id} - F_0$ is a homeomorphism from $X_{r_q}$ to its image thanks to the fact that $F_0$ has Lipschitz constant less than 1 on $X_{r_q}$. From the above identity we deduce that the map 
\[
\xi_0: [\lambda_d,1]^A \rightarrow \ell^{d+1}(A^c)
\]
is continuous.

Let $x_1\in [\lambda_d,1]^A$. The function $x=\xi_0(x_1)\sqcup x_1$ is a solution of \eqref{bdrylaw} if and only if $x_1$ satisfies the equation
\begin{equation}
\label{bdrylawA}
x_1 = \bigl( Q*(\xi_0(x_1)^d \sqcup x_1^d)\bigr)\bigr|_A,
\end{equation}
which can be rewritten as
\begin{equation}
\label{bdrylawA2}
\varphi_d(x_1) = \bigl( q*(\xi_0(x_1)^d \sqcup x_1^d)\bigr)\bigr|_A.
\end{equation}
We set
\begin{equation}
\label{mud}
\mu_d:= \max_{r\in [0,+\infty)} \varphi_d(r) = \varphi_d(\lambda_d) = \lambda_d {\textstyle \left( 1- \frac{1}{d} \right) },
\end{equation}
and claim that 
\begin{equation}
\label{claim} 
0 \leq q*(\xi_0(x_1)^d \sqcup x_1^d) < \mu_d \qquad \mbox{on } A,
\end{equation}
for every $x_1\in  [\lambda_d,1]^A$. The first inequality is clear. In order to prove the second one, we use the upper bound
\[
\begin{split}
\|q* ( \xi_0(x_1)^d &\sqcup x_1^d)\|_{\infty} \leq \|q\|_{d+1} \| \xi_0(x_1)^d \sqcup x_1^d\|_{\frac{d+1}{d}} = \|q\|_{d+1}  \| \xi_0(x_1) \sqcup x_1 \|_{d+1}^d \\ &= \|q\|_{d+1}  \bigl( \|\xi_0(x_1)\|_{d+1}^{d+1} + \|x_1\|_{d+1}^{d+1} \bigr)^{\frac{d}{d+1}} \leq  \|q\|_{d+1} ( r_q^{d+1} + |A| )^{\frac{d}{d+1}} \\ &= \|q\|_{d+1} ( r_q^{d+1} + n )^{\frac{d}{d+1}}.
\end{split}
\]
By the bound $\|q\|_{d+1}\leq \|q\|_{\frac{d+1}{2}}$ and our choice of $r_q$ we have
\begin{equation}
\label{mir}
\begin{split}
\|q* ( \xi_0(x_1)^d \sqcup x_1^d)\|_{\infty} &\leq \|q\|_{\frac{d+1}{2}} ( r_q^{d+1} + n )^{\frac{d}{d+1}} = f_{d,n}(r_q) ( r_q^{d+1} + n)^{\frac{d}{d+1}} \\ &= \varphi_d(r_q) < \varphi_d(\lambda_d) = \mu_d,
\end{split}
\end{equation}
as claimed in \eqref{claim}. Thanks to \eqref{claim}, we can rewrite \eqref{bdrylawA2} as
\[
x_1 = \psi\bigl( \bigl( q*(\xi_0(x_1)^d \sqcup x_1^d)\bigr)\bigr|_A \bigr),
\]
where $\psi: [0,\mu_d] \rightarrow [\lambda_d,1]$ is the inverse of the restriction of $\varphi_d$ to the interval $[\lambda_d,1]$, on which $\varphi_d$ is strictly monotonically decreasing. Therefore, $x_1\in [\lambda_d,1]^A$ satisfies \eqref{bdrylawA} if and only if it is a fixed point of the continuous map
\begin{equation}
\label{eq: MapG}
G: [\lambda_d,1]^A \rightarrow  [\lambda_d,1]^A, \qquad G(x_1) := \psi\bigl( \bigl( q*(\xi_0(x_1)^d \sqcup x_1^d)\bigr)\bigr|_A \bigr).
\end{equation}
By Brouwer's fixed point theorem, $G$ has a fixed point $\overline{x}_1$ on the $n$-dimensional cube $[\lambda_d,1]^A$. 
Setting $\overline{x}_0:= \xi_0(\overline{x}_1)$, we obtain that $\overline{x}:= \overline{x}_0\sqcup \overline{x}_1$ is a positive solution of \eqref{bdrylaw} and as such takes values in $(0,1)$.

By the strict upper bound in \eqref{claim} and by the properties of $\psi$, we have $\overline{x}_1 > \lambda_d$. Since $\overline{x}_0$ belongs to $X_{r_q}$ with $r_q< \rho(d,n)$, we have
\[
\|\overline{x}_0\|_{\infty} \leq \|\overline{x}_0\|_{d+1} < \rho(d,n).
\]
We conclude that $\overline{x}\in \ell^{d+1}(S)$ is a positive solution of \eqref{aim}. 

There remains to prove the bounds (i), (ii) and (iii). By construction,
\begin{equation}
\label{sharp}
\|\overline{x}|_{A^c}\|_{d+1} = \|\overline{x}_0\|_{d+1} \leq r_q.
\end{equation}
Since $f_{d,n}$ is strictly increasing and concave on $[0,\rho]$, its restriction to $[0,\rho]$ has an inverse which is strictly increasing and convex on $[0,\eta(d,n)]$. The convexity of this inverse implies the inequality
\[
r_q = \bigl( f_{d,n}|_{[0,\rho]} \bigr)^{-1}\bigl(\|q\|_{\frac{d+1}{2}}\bigr) \leq \frac{\rho(d,n)}{\eta(d,n)} \|q\|_{\frac{d+1}{2}},
\]
so (i) follows from \eqref{sharp}. 

Since the restriction of $\varphi_d$ to $[\lambda_d,1]$ is strictly decreasing and concave, its inverse $\psi_d$ is strictly decreasing and concave on $[0,\mu_d]$. The concavity of $\psi_d$ implies
\begin{equation}
\label{conc}
\psi_d(s) \geq 1 - \frac{1-\lambda_d}{\mu_d} s = 1- \frac{d^{\frac{d}{d-1}}-d}{d-1} s \qquad \forall s\in [0,\mu_d].
\end{equation}
By the first inequality in \eqref{mir}, we have
\[
\|q* \overline{x}^d \|_{\infty} \leq  \|q\|_{\frac{d+1}{2}} ( r_q^{d+1} + n )^{\frac{d}{d+1}} \leq  ( 1 + n )^{\frac{d}{d+1}} \|q\|_{\frac{d+1}{2}} .
\]
Together with the fact that $\psi_d$ is decreasing and satisfies the concavity inequality \eqref{conc}, the above upper bound implies
\[
\overline{x}|_A = \overline{x}_1 = \psi_d \bigl( (q* \overline{x}^d)|_A \bigr) \geq 1 - {\textstyle \frac{d^{\frac{d}{d-1}}-d}{d-1}} ( 1 + n )^{\frac{d}{d+1}} \|q\|_{\frac{d+1}{2}} ,
\]
proving (ii). 

The map $F_{\overline{x}_1}$ is monotonically increasing on the subset of non-negative functions in $\ell^{d+1}(A^c)$, with respect to the standard partial order of functions. Since the fixed point $\overline{x}_0$ of $F_{\overline{x}_1}$ satisfies $\overline{x}_0\geq 0$, we have
\[
\overline{x}|_{A^c} = \overline{x}_0 = F_{\overline{x}_1}(\overline{x}_0) \geq  F_{\overline{x}_1}(0) = \bigl( Q * (0_{A^c} \sqcup \overline{x}_1^d) \bigr)|_{A^c}.
\]
By evaluating at $i\in A^c$ and using (ii), we obtain the lower bound
\[
\begin{split}
\overline{x} (i) &\geq \bigl( Q * (0_{A^c} \sqcup \overline{x}_1^d) \bigr) (i) = \sum_{j\in A} Q(i-j) \overline{x}^d_1(j) \\ &\geq \left( 1 - {\textstyle \frac{d^{\frac{d}{d-1}}-d}{d-1}} ( 1 + n )^{\frac{d}{d+1}} \|q\|_{\frac{d+1}{2}} \right)^d \sum_{j\in A} Q(i-j) \\ &\geq \left( 1 - {\textstyle \frac{d}{d-1}} (d^{\frac{d}{d-1}}-d) ( 1 + n )^{\frac{d}{d+1}} \|q\|_{\frac{d+1}{2}} \right) \sum_{j\in A} Q(i-j), 
\end{split}
\]
where in the last step we have used the Bernoulli inequality. This proves (iii) and concludes the proof of Proposition \ref{existence}.
\end{proof}

\begin{rk}
\label{uniqueness1} 
{\rm ({\em Uniqueness for $|A|=1$}) Consider the standard partial order on the space of real valued functions. 
It is easy to show that the map $G:[\lambda_d,1]^A \rightarrow [\lambda_d,1]^A$ is monotonically decreasing. Indeed, the fact that the map $(x_0,x_1) \mapsto F_{x_1}(x_0)$ is monotonically increasing implies that if $x_1\leq x_1'$ then $F_{x_1}^n(x_0) \leq F_{x_1'}^n(x_0)$ for every $n\in \N$ and every $x_0$. Taking the limit in $n$, we deduce that the map $\xi_0$ which associates to every $x_1\in [\lambda_d,1]^A$ the unique fixed point of $F_{x_1}$ is also monotonically increasing, and so is the map $x_1 \mapsto (q*( \xi_0(x_1) \sqcup x_1^d))|_A \in [0,\mu_d]^A$. From the fact that the function $\psi_d$ is montonically decreasing on $[0,\mu_d]$, we deduce that $G$ is monotonically decreasing, as claimed.
When $|A|=1$, this implies that $G$ has a unique fixed point . In this case, the solution $\overline{x}$ of \eqref{aim} is unique.
} 
\end{rk}

\begin{rk}
\label{uniqueness1b}
{\rm ({\em Uniqueness for $|A|>1$})
If $n=|A|>1$, we do not know whether the solution of \eqref{aim} is unique under the assumption $\|q\|_{\frac{d+1}{2}} \leq \eta(d,n)$. By assuming a stronger smallness assumption on $\|q\|_{\frac{d+1}{2}}$, we surely have existence and uniqueness of a solution $x\in \ell^{d+1}(S)$ of the equation $x=Q*x^d$ which is sufficiently close to $\mathbbm{1}_A$ in the $(d+1)$-norm. This follows from the implicit mapping theorem applied to the continuously differentiable map
\[
H: \ell^{\frac{d+1}{2}}(S) \times \ell^{d+1}(S) \rightarrow \ell^{d+1}(S), \qquad H(Q,x) = x - Q* x^d.
\]
Indeed, $H(\mathbbm{1}_{\{0\}}, \mathbbm{1}_A) = 0$ and the differential of $H$ with respect to the second variable at $(\mathbbm{1}_{\{0\}}, \mathbbm{1}_A)$ is the linear operator
\[
d_2 H (\mathbbm{1}_{\{0\}}, \mathbbm{1}_A)[y] = \bigl( 1 - d \mathbbm{1}_A \bigr) y,
\]
which is an isomorphism on $\ell^{d+1}(S)$ because $d\neq 1$. From Theorem \ref{thm: Zachary} and the first identity in Proposition \ref{useful}, we then deduce that there exists positive numbers $\eta'(d,n)$ and $\delta(d,n)$ such that if $\|q\|_{\frac{d+1}{2}} \leq \eta'(d,n)$ and $A\subset S$ has $n$ elements, then the Markovian gradient specification which is induced by $Q$ on the regular $d$-tree with local state space $S$ has a unique spatially homogeneous Markov-chain Gibbs measure $\mu$ whose single-site marginal probability distribution $\pi_{\mu}$ satisfies
\[
\bigl\|\pi_{\mu} - {\textstyle \frac{1}{|A|}} \mathbbm{1}_A \bigr\|_1 < \delta(d,|A|),
\]
as claimed in Remark \ref{uniqueness0}. The bounds $\eta'$ and $\delta$ which one gets from standard quantitative versions of the implicit function theorem are much worse than the ones appearing in Theorem \ref{main1-thm}. In order to obtain better bounds, one can look for a solution $x\in \ell^{d+1}(S)$ of the equation $x=Q*x^d$ which is close to $\mathbbm{1}_A$ by considering the
fixed point of the map
\[
(x_0,x_1) \mapsto \Bigl( \bigl(Q* (x_0^d \sqcup x_1^d)\bigr)|_{A^c}, \psi_d \bigl( \bigl(q* (x_0^d \sqcup x_1^d)\bigr)|_{A} \bigr) \Bigr),
\]
which can be shown to be a contraction on a suitable closed subset of $\ell^{d+1}(A^c) \times \ell^{\infty}(A)$ if $\|q\|_{\frac{d+1}{2}}$ is small enough. In this way, one gets an existence and uniqueness statement as above but with bounds which are not too much worse than those in Theorem \ref{main1-thm}.}
\end{rk}

\subsection{Proof of Lemma  \ref{confronto} and of the bounds on \texorpdfstring{$\eta$}{eta}}
\label{lemmas-sec}

For the sake of simplicity,  we omit subindices and use the abbreviations $f=f_{d,n}$, $g=g_d$, $\varphi=\varphi_d$, $\lambda=\lambda_ds$ throughout this section.

\begin{proof}[Proof of Lemma \ref{confronto}]
 The identity
\[
f'(r) = \frac{(1-d)r^{d+1}-dnr^{d-1}+n}{(r^{d+1}+n)^{\frac{d}{d+1}+1}}
\]
shows that $f$ is strictly increasing on the interval $[0,\rho]$ and strictly decreasing on the interval $[\rho,+\infty)$, where $\rho=\rho(d,n)$ is the unique positive solution of \eqref{defrho}. Since
\[
f'(\lambda) = \frac{(1-d)\lambda^2}{d(r^{d+1}+n)^{\frac{d}{d+1}+1}}
\]
is negative, the number $\rho$ at which $f$ achieves its global maximum belongs to the interval $(0,\lambda)$.

From the identity
\[
f''(r) = \frac{(d-1)r^{2d+1}+3ndr^{2d-1} - n(d^2+d+1)r^d-n^2 d(d-1)r^{d-1}}{(r^{d+1}+n)^{\frac{d}{d+1}+2}},
\]
and the Decartes rule of signs, we deduce that $f''$ changes sign exactly once on $(0,+\infty)$, so the fact that the number
\[
f''(1) = - \frac{(d-1)(n^2 d + (d-1)n-1)}{(r^{d+1}+1)^{\frac{d}{d+1}+2}} 
\]
is negative implies that $f$ is strictly concave on $[0,1]$.

The function $g$ is positive on $(0,\lambda)$, where it strictly decreases from $+\infty$ to $g(\lambda)=0$. Since $f(0)=0$ and $f(\lambda)>0$, there exist solutions $r_*\in (0,\lambda)$ of the equation $f(r_*)=g(r_*)$. Since $g$ is strictly convex on $(0,\lambda]$ and $f$ is strictly concave on this interval, the solution $r_*$ is unique and we have
\begin{equation}
\label{dico}
g>f \quad \mbox{on } (0,r_*), \qquad g<f \quad \mbox{on } (r_*,\lambda].
\end{equation}
There remains to prove that $r_*$ is larger than $\rho$. Thanks to \eqref{dico}, it sufficies to find a number $\sigma\in [\rho,\lambda)$ such that $f(\sigma)<g(\sigma)$. 

We set $\rho=\lambda(1-\theta)$, where $\theta$ belongs to the interval $(0,1)$, and rewrite \eqref{defrho} as
\[
n = \frac{(d-1)\lambda^2}{d} \frac{(1-\theta)^{d+1}}{1-(1-\theta)^{d-1}}.
\]
From the Bernoulli inequality
\[
(1+x)^p \geq 1 + px \qquad \forall x\geq -1, \; \forall p\geq 1,
\]
we deduce the bound
\[
n\geq \frac{(d-1)\lambda^2}{d} \frac{1-(d+1)\theta}{(d-1)\theta} = \frac{\lambda^2}{d}  \frac{1-(d+1)\theta}{\theta},
\]
which can be reformulated as
\[
\theta \geq \frac{\lambda^2}{dn+\lambda^2 (d+1)}.
\]
We conclude that $\rho\leq \sigma$ where
\[
\sigma := \left( 1 - \frac{\lambda^2}{dn+\lambda^2 (d+1)} \right) \lambda =  \left( 1 + \frac{\lambda^2}{d(n+\lambda^2)} \right)^{-1} \lambda < \lambda.
\]
In the remaining part of the proof, we show that $f(\sigma) < g(\sigma)$. Using again the Bernoulli inequality we find
\[
g(\sigma) = \left( \Bigl( 1 + \frac{\lambda^2}{d(n+\lambda^2)} \Bigr)^{\frac{d^2-1}{2}} - 1 \right)^{\frac{2}{d+1}}  \geq \left( \frac{d^2-1}{2d} \frac{\lambda^2}{n+\lambda^2} \right)^{\frac{2}{d+1}},
\]
so it is enough to prove the inequality
\begin{equation}
\label{lbg}
f(\sigma) < \left( \frac{d^2-1}{2d} \frac{\lambda^2}{n+\lambda^2} \right)^{\frac{2}{d+1}}.
\end{equation}
We first deal with the case $d=2$, in which $\lambda$ and $\sigma$ have the values
\[
\lambda= \frac{1}{2}, 
\qquad \sigma = \frac{4n+1}{8n+3}.
\]
Since
\[
f(\sigma) = \frac{\sigma-\sigma^2}{(\sigma^3+ n)^{\frac{2}{3}}}< \frac{\sigma-\sigma^2}{n^{\frac{2}{3}}} = \frac{16n^2+12 n +2}{n^{\frac{2}{3}} (8n+3)^2},
\]
\eqref{lbg} will be proven if we can show that
\[
\frac{16n^2+12 n +2}{n^{\frac{2}{3}} (8n+3)^2} < \left( \frac{3}{4} \frac{1}{4n+1} \right)^{\frac{2}{3}}.
\]
By raising both sides to the power $3$, the above inequality is easily seen to be equivalent to 
\begin{equation}
\label{z1}
2^7 (4n+1)^2 (8n^2+6n+1)^3 < 9 n^2 (8n+3)^6.
\end{equation}
Since
\[
9 n^2 (8n+3)^6 > 8 n^2 (64 n^2+48n+ 9)^3 >  8 n^2 (64 n^2+48n+ 8)^3 = 2^{12} n^2 (8n^2+6n+1)^3,
\]
\eqref{z1} is implied by the inequality
\[
(4n+1)^2 < 32 n^2,
\]
which is indeed true for every $n\geq 1$, being equivalent to
\[
(4n-1)^2>2.
\]
This proves \eqref{lbg} in the case $d=2$. The case $d\geq 3$ can be dealt with by starting from the weaker bound
\[
f(\sigma) = \frac{\sigma-\sigma^d}{(\sigma^{d+1} + n)^{\frac{d}{d+1}}} < \frac{\sigma}{n^{\frac{d}{d+1}}} = \frac{\lambda}{n^{\frac{d}{d+1}}} \left( 1 + \frac{\lambda^2}{d(n+\lambda^2)} \right)^{-1}.
\]
By the above upper bound on $f(\sigma)$, \eqref{lbg} holds true if we can prove the inequality
\begin{equation}
\label{z2}
\left( \frac{d^2-1}{2d} \frac{\lambda^2}{n+\lambda^2} \right)^2 > \frac{\lambda^{d+1}}{n^d} \left( 1 + \frac{\lambda^2}{d(n+\lambda^2)} \right)^{-(d+1)}.
\end{equation}
Using the identity $\lambda=d^{-\frac{1}{d-1}}$ and the Bernoulli inequality, the right-hand side of \eqref{z2} can be estimated in the following way:
\[
\begin{split}
\frac{\lambda^{d+1}}{n^d} & \left( 1 + \frac{\lambda^2}{d(n+\lambda^2)} \right)^{-(d+1)} = \frac{\lambda^2}{d n^d} \left( 1 + \frac{\lambda^2}{d(n+\lambda^2)} \right)^{-(d+1)} \\ &\leq \frac{\lambda^2}{d n^d} \left( 1 + \frac{(d+1)\lambda^2}{d(n+\lambda^2)} \right)^{-1} = \frac{\lambda^2}{n^d} \cdot \frac{n+\lambda^2}{dn+(2d+1)\lambda^2}.
\end{split}
\]
Therefore, \eqref{z2} is implied by
\begin{equation}
\label{z3}
\left( \frac{d^2-1}{2d} \frac{\lambda^2}{n+\lambda^2} \right)^2 >  \frac{\lambda^2}{n^d} \cdot \frac{n+\lambda^2}{dn+(2d+1)\lambda^2}.
\end{equation}
A simple algebraic manipulation
  shows that  \eqref{z3} is equivalent to
\begin{equation}
\label{z4}
d^2 p_{d,n}(d) + (n+2\lambda^2)d + \lambda^2 >0,
\end{equation}
where
\[
p_{d,n}(x) := (n+2\lambda^2) x^3 + \lambda^2 x^2 - 2 (n+2\lambda^2) x - 2 \left( \lambda^2 + \frac{2(n+\lambda^2)^3}{\lambda^2 n^d} \right).
\]
We shall prove that $p_{d,n}(x)> 0$ for every $x\geq 3$, $d\geq 3$ and $n\geq 1$, which implies \eqref{z4} and by the above discussion \eqref{lbg} for every $d\geq 3$. By the Decartes rule of signs, the polynomial $p_{d,n}$ has precisely one positive real root $\alpha$, and 
\[
p_{d,n}(x)<0 \quad \mbox{on }[0,\alpha), \qquad p_{d,n}(x)>0 \quad \mbox{on }(\alpha,+\infty).
\]
Therefore, it is enough to prove
\begin{equation}
\label{z5}
p_{d,n}(3) >0 \qquad \forall d\geq 3, \; \forall n\geq 1.
\end{equation}
Using the inequalities $n\geq 1$ and $d\geq 3$, we find
\begin{equation}
\label{qif}
\begin{split}
p_{d,n}(3) &= 21 n + 49 \lambda^2 - 4 \frac{(n+\lambda^2)^3}{\lambda^2 n^d} \geq 21 + 49 \lambda^2 - 4 \frac{(n+\lambda^2)^3}{\lambda^2 n^3} \\
&= 21 + 49\lambda^2 - \frac{4}{\lambda^2} \left( 1 + \frac{\lambda^2}{n} \right)^3 \geq 21 + 49\lambda^2 - \frac{4}{\lambda^2} \left( 1 + \lambda^2 \right)^3 \\ &= 9 - \frac{4}{\lambda^2} + 37 \lambda^2 - 4 \lambda^4.
\end{split}
\end{equation}
From the fact that the sequence $\lambda_d = d^{-\frac{1}{d-1}}$ is monotonically increasing and converges to 1 we deduce
\[
\frac{1}{3} \leq \lambda_d^2 < 1 \qquad \forall d\geq 3.
\]
Therefore, \eqref{qif} implies
\[
p_{d,n}(3) \geq 9 - 12 + \frac{37}{3} - 4 = \frac{16}{3}>0,
\]
as we wished to prove.
\end{proof}
 
We conclude this section by proving the bounds \eqref{bb0} and \eqref{bbb0} for the quantity $\eta(d,n)$. 

\begin{lemma}
\label{asymptotics}
For every pair of integers $d\geq 2$ and $n\geq 1$  the quantity $\eta(d,n)$ satisfies the bounds
\begin{align}
\label{bb}
d^{-\frac{1}{d-1}} \left( 1 -{\textstyle \frac{1}{d}}  \right) (n+1)^{- \frac{d}{d+1}} \leq \; &\eta(d,n) \leq d^{-\frac{1}{d-1}} \left( 1 - {\textstyle \frac{1}{d}}  \right) n^{- \frac{d}{d+1}}, \\
\label{bbb}
\underline{c} \, d\, n^d \leq \; &\eta(d,n)^{-(d+1)} \leq \overline{c}  \, d\, n^d,
\end{align}
for suitable positive numbers $\underline{c}$ and $\overline{c}$.
\end{lemma}

\begin{proof}
By \eqref{mud}, the function $\varphi(r)=r-r^d$ achieves its maximum at $\lambda=d^{-\frac{1}{d-1}}$, where it has the value $\lambda(1-\frac{1}{d})$. From this fact, we obtain the bound
\[
\eta(d,n) = \frac{\rho-\rho^d}{(\rho^{d+1}+n)^{\frac{d}{d+1}}} \leq \lambda \left(1-{\textstyle \frac{1}{d}} \right) n^{- \frac{d}{d+1}},
\]
which gives us the right-hand side estimate in \eqref{bb}. From the fact that $\eta(d,n)$ is the maximum of the function $f$ and that $\lambda$ is smaller than one, we obtain also the lower bound
\[
\eta(d,n) \geq  \frac{\lambda-\lambda^d}{(\lambda^{d+1}+n)^{\frac{d}{d+1}}} = \frac{\lambda {\textstyle \left( 1 - \frac{1}{d} \right)}}{(\lambda^{d+1}+n)^{\frac{d}{d+1}}}  \geq d^{-\frac{1}{d-1}} {\textstyle \left( 1 - \textstyle{\frac{1}{d}} \right)} (n+1)^{- \frac{d}{d+1}},
\]
which is the left-hand side estimate in \eqref{bb}.

The sequence
\[
\lambda_d = d^{-\frac{1}{d-1}}= e^{-\frac{1}{d-1} \log d}
\]
is increasing from the value $\frac{1}{2}$ it takes for $d=2$ towards the value $1$ of its limit for $d\rightarrow \infty$.  From this fact and the identity
\[
 (\lambda - \lambda^d)^{d+1} = \lambda^{d+1} \left( 1- {\textstyle \frac{1}{d}}  \right)^{d+1} = \frac{\lambda^2}{d} \left( 1- {\textstyle \frac{1}{d}}  \right)^{d+1} 
 \]
 we deduce that
 \begin{equation}
 \label{bbb1}
  \frac{1}{32\, d}\leq (\lambda - \lambda^d)^{d+1} \leq \frac{1}{e\,d},
  \end{equation}
 where we have used also the fact that the sequence $\left( 1- \frac{1}{d} \right)^{d+1}$ is increasing from the value $\frac{1}{8}$ it takes for $d=2$ to the value $\frac{1}{e}$ of its limit. Moreover, we have
 \begin{equation}
 \label{bbb2}
 (\lambda^{d+1}+n)^d= n^d {\textstyle \left( 1 + \frac{\lambda^2}{dn} \right)^d \leq n^d \left( 1 + \frac{1}{dn} \right)^d \leq n^d \left( 1 + \frac{1}{d} \right)^d} \leq e\, n^d,
 \end{equation}
 where we have used the fact that the sequence $\left( 1+ \frac{1}{d} \right)^{d}$ is increasing and converges to $e$. Since $\eta(d,n)$ is the maximum of $f$, \eqref{bbb1} and \eqref{bbb2} imply
  \begin{equation}
 \label{bbb3}
 \eta(d,n)^{d+1} \geq f(\lambda)^{d+1} = \frac{ (\lambda - \lambda^d)^{d+1}}{(\lambda_d^{d+1}+n)^d} \geq \frac{1}{32 \,d\, e\,n^d}.
 \end{equation}
On the other hand, since $\lambda$ maximizes the function $\varphi$, we have by \eqref{bbb1}
  \begin{equation}
 \label{bbb4}
\eta(d,n)^{d+1}  \leq \frac{(\lambda - \lambda^d)^{d+1}}{n^d} \leq \frac{1}{e\, d\, n^d}.
 \end{equation}
 The bounds \eqref{bbb3} and \eqref{bbb4} imply that \eqref{bbb} holds with $\underline{c}=e$ and $\overline{c}=32e$.
\end{proof}

\subsection{Proof of Theorem \ref{main1-thm}}
\label{proof-sec}

Building on Theorem \ref{thm: Zachary} and Proposition \ref{useful}, we now show how Theorem \ref{main1-thm} follows from Proposition \ref{existence}. We assume that
\[
\|q\|_{\frac{d+1}{2}} = \|Q- \mathbbm{1}_{\{0\}} \|_{\frac{d+1}{2}} \leq \eta(d,N)
\]
and we fix a subset $A$ with $1\leq n:=|A| \leq N$. Let $\overline{x}\in \ell^{d+1}(S)$ be a solution of \eqref{aim}, whose existence is guaranteed by Proposition \ref{existence}. Let $
u = \overline{x}^d \in \ell^{\frac{d+1}{d}}(S)$ be the corresponding solution of the boundary law equation \eqref{eq: boundarylaw} with $c=1$. By Theorem \ref{thm: Zachary} and Proposition \ref{useful}, the boundary law $u$ induces a spatially homogeneous Markov-chain Gibbs measure $\mu\in \mathcal{MG}(\gamma)$ whose single-site marginal distribution $\pi_{\mu}$ and transition matrix $P_{\mu}$ are given by
\[
\pi_{\mu}(i) = \frac{u(i)^{\frac{d+1}{d}}}{\|u^{\frac{d+1}{d}}\|}, \qquad P_{\mu}(i,j) = \frac{u(j) Q(i-j)}{(Q* u) (i)}.
\]
From the identities $u=\overline{x}^d$ and $\overline{x} = Q * \overline{x}^d$, we find
\begin{equation}
\label{translate}
\pi_{\mu}(i) = \frac{\overline{x}(i)^{d+1}}{\|\overline{x}\|_{d+1}^{d+1}}, \qquad P_{\mu}(i,j) = \frac{\overline{x}(j)^d Q(i-j)}{\overline{x}(i)}.
\end{equation}
Then \eqref{aim} implies, setting $\theta:= \theta(d,n) := (\frac{\rho(d,n)}{\lambda_d} )^{d+1}$, 
\[
\|\pi_{\mu}|_{A^c}\|_1 = \frac{\|\overline{x}|_{A^c}\|_{d+1}^{d+1}}{\|\overline{x}\|_{d+1}^{d+1}} < \frac{\rho(d,n)^{d+1}}{\|\overline{x}\|_{d+1}^{d+1}} < \theta  \frac{ \left(\min_A \overline{x} \right)^{d+1}}{\|\overline{x}\|_{d+1}^{d+1}}  = \theta \min_A \frac{\overline{x}^{d+1}}{\|\overline{x}\|_{d+1}^{d+1}} = \theta  \min_A \pi_{\mu},
\]
proving \eqref{main2} in Theorem \ref{main1-thm}. By \eqref{translate}, the diagonal elements $\Delta(i)=P_{\mu}(i,i)$ of the transition matrix are given by
\[
\Delta(i) = \overline{x}(i)^{d-1}.
\]
The bounds
\begin{equation}
\label{primi}
\|\overline{x}|_{A^c}\|_{\infty} \leq \|\overline{x}|_{A^c}\|_{d+1} < \rho(d,n) < \lambda_d = d^{- \frac{1}{d-1}} < \overline{x}|_A < 1,
\end{equation}
from \eqref{aim} translate into
\[
\|\Delta_{A^c}\|_{\infty} \leq \|\Delta|_{A^c}\|_{\frac{d+1}{d-1}} < \rho(d,n)^{d-1}  <  \frac{1}{d} < \Delta|_A < 1,
\]
proving \eqref{main1} in Theorem \ref{main1-thm}.

\begin{rk}
\label{uniqueness2}
{\rm As shown in Remark \ref{uniqueness1}, if $|A|=1$ then the solution $\overline{x}\in \ell^{d+1}(S)$ of \eqref{aim} is unique. Together with the uniqueness of the boundary law which is determined by a Gibbs measure (see again Theorem \ref{thm: Zachary}), this implies that in the case $|A|=1$ the above $\mu$ is the unique spatially homogeneous Markov-chain Gibbs measure $\mu$ whose transition matrix 
$P_{\mu}$ satisfies \eqref{main1}.}
\end{rk}

In the following proof of statements (i)-(v) of Theorem \ref{main1-thm}, we use the abbreviations
\[
\epsilon := \|q\|_{\frac{d+1}{2}} =  \|Q- \mathbbm{1}_{\{0\}} \|_{\frac{d+1}{2}}, \qquad \rho = \rho(d,n), \qquad \eta = \eta(d,n). 
\]
By statement (i) in Proposition \ref{existence} we have
\[
\|\Delta|_{A^c}\|_{\frac{d+1}{d-1}} = \|\overline{x}|_{A^c}\|_{d+1}^{d-1} \leq \frac{\rho^{d-1}}{\eta^{d-1}} \epsilon^{d-1},
\]
proving assertion (i) in Theorem \ref{main1-thm}. By statement (ii) in Proposition \ref{existence} we have, using the Bernoulli inequality,
\[
\Delta|_A = \overline{x}_A^{d-1} \geq \left( 1 - \frac{d^{\frac{d}{d-1}}-d}{d-1} (1+n)^{\frac{d}{d+1}} \epsilon \right)^{d-1} \geq 1 - \left( d^{\frac{d}{d-1}} - d \right) (1+n)^{\frac{d}{d+1}} \epsilon,
\]
which proves statement (ii) in Theorem \ref{main1-thm}.

There remains to prove the bounds (iii), (iv) and (v) on the single-site marginal distribution $\pi_{\mu}$. By \eqref{primi}, the $(d+1)$-norm of $\overline{x}$ has the lower bound
\[
\|\overline{x}\|_{d+1}^{d+1} \geq \|\overline{x}|_A\|_{d+1}^{d+1} \geq  \lambda_d^{d+1} |A|.
\]
By statement (i) in Proposition \ref{existence} we have
\[
\| \pi_{\mu}|_{A^c} \|_1 = \frac{\|\overline{x}|_{A^c}\|_{d+1}^{d+1}}{\|\overline{x}\|_{d+1}^{d+1}} \leq \frac{ \left( \frac{\rho}{\eta} \right)^{d+1} \, \epsilon^{d+1}}{\lambda_d^{d+1} |A|} = \frac{d^{\frac{d+1}{d-1}} \rho^{d+1}}{\eta^{d+1} n } \epsilon^{d+1},
\]
which proves assertion (iii) in Theorem \ref{main1-thm}. From statement (ii) in Proposition \ref{existence} we obtain, using the Bernoulli inequality,
\[
\begin{split}
\|\overline{x}\|_{d+1}^{d+1} \geq \|\overline{x}|_A\|_{d+1}^{d+1}  &\geq \left( 1 - \frac{d^{\frac{d}{d-1}}-d}{d-1} (1+n)^{\frac{d}{d+1}} \epsilon \right)^{d+1} |A| \\ & \geq  \left( 1 - \frac{d+1}{d-1} \bigl( d^{\frac{d}{d-1}}-d\bigr)  (1+n)^{\frac{d}{d+1}} \epsilon \right) |A|, 
\end{split}
\]
and hence
\[
\pi_{\mu}|_A = \frac{\overline{x}|_A^{d+1}}{ \|\overline{x}\|_{d+1}^{d+1}}  \leq \frac{1}{\|\overline{x}\|_{d+1}^{d+1}} \leq  \left( 1 - \frac{d+1}{d-1} \bigl( d^{\frac{d}{d-1}}-d\bigr)  (1+n)^{\frac{d}{d+1}} \epsilon \right)^{-1} \frac{1}{|A|},
\]
which proves the right-hand side inequality in statement (iv) of Theorem \ref{main1-thm}. On the other hand, using statement (i) in Proposition \ref{existence} we have
\[
\begin{split}
\|\overline{x}\|_{d+1}^{d+1} = \|\overline{x}|_A\|_{d+1}^{d+1} + \|\overline{x}_{A^c}\|_{d+1}^{d+1} &\leq |A| + \frac{\rho^{d+1}}{\eta^{d+1}}  \epsilon^{d+1} \\ & = \left( 1 +  \frac{\rho^{d+1}}{n \eta^{d+1}}  \epsilon^{d+1}  \right) |A|,
\end{split} 
\]
and hence, using also the inequality $\epsilon\leq \eta(d,N) \leq 1$ (see \eqref{bb}),
\begin{equation}
\label{massa}
\|\overline{x}\|_{d+1}^{-(d+1)}  \geq \left( 1 -  \frac{\rho^{d+1}}{n \eta^{d+1}}  \epsilon^{d+1}  \right) \frac{1}{|A|} \geq \left( 1 - \frac{\rho^{d+1}}{n \eta^{d+1}}  \epsilon  \right) \frac{1}{|A|} . 
\end{equation}
Using again statement (ii) in Proposition \ref{existence} and the Bernoulli inequality we obtain
\[
\begin{split}
\pi_{\mu}|_A &= \frac{\overline{x}|_A^{d+1}}{ \|\overline{x}\|_{d+1}^{d+1}}  \geq \left( 1 - \frac{d^{\frac{d}{d-1}}-d}{d-1} (1+n)^{\frac{d}{d+1}} \epsilon \right)^{d+1} \left( 1 - \frac{\rho^{d+1}}{n \eta^{d+1}}  \epsilon  \right) \frac{1}{|A|}  \\ &\geq \left( 1 - \frac{d+1}{d-1}\bigl( d^{\frac{d}{d-1}}-d \bigr) (1+n)^{\frac{d}{d+1}} \epsilon \right)  \left( 1 - \frac{\rho^{d+1}}{n \eta^{d+1}}  \epsilon  \right)  \frac{1}{|A|} \\ &\geq \left( 1 -  \Bigl( \frac{d+1}{d-1}\bigl( d^{\frac{d}{d-1}}-d \bigr) (1+n)^{\frac{d}{d+1}} + \frac{\rho^{d+1}}{n \eta^{d+1}}  \Bigr) \epsilon  \right) \frac{1}{|A|}.
\end{split}
\]
This proves the left-hand inequality in statement (iv) of Theorem \ref{main1-thm}. By \eqref{massa}, statement (iii)  in Proposition \ref{existence} and a last application of the Bernoulli inequality, we obtain for every $i\in A^c$ the lower bound
\[
\begin{split}
\pi_{\mu}(i) &\geq \frac{1}{|A|}  \left( 1 - {\textstyle \frac{\rho^{d+1}}{n \eta^{d+1}}}  \epsilon  \right)  \bigl(1- {\textstyle \frac{d}{d-1}} ( d^{\frac{d}{d-1}}-d) (1+n)^{\frac{d}{d+1}} \epsilon \bigr)^{d+1} \Bigl( \sum_{j\in A} Q(i-j) \Bigr)^{d+1} \\
& \geq  \frac{1}{|A|} \left( 1 -{\textstyle \frac{\rho^{d+1}}{n \eta^{d+1}} } \epsilon  \right)  \bigl(1- {\textstyle \frac{d(d+1)}{d-1}} ( d^{\frac{d}{d-1}}-d) (1+n)^{\frac{d}{d+1}} \epsilon \bigr) \Bigl( \sum_{j\in A} Q(i-j)  \Bigr)^{d+1} 
\\ 
& \geq  \frac{1}{|A|} \left( 1 - \bigl( {\textstyle \frac{\rho^{d+1}}{n \eta^{d+1}} }  + {\textstyle \frac{d(d+1)}{d-1}} ( d^{\frac{d}{d-1}}-d) (1+n)^{\frac{d}{d+1}}\bigr)  \epsilon \right) \Bigl( \sum_{j\in A} Q(i-j)  \Bigr)^{d+1}.
\end{split}
\]
This proves statement (v) of Theorem \ref{main1-thm} and concludes the proof of Theorem \ref{main1-thm}.

%\newpage
\printbibliography

\end{document}